\documentclass[12pt]{amsart}
\usepackage[margin=1in]{geometry} 
\usepackage{amsmath,amsthm,amssymb,amsfonts}
\usepackage[shortlabels]{enumitem}
\usepackage{bbm}
\usepackage{graphicx}
\usepackage[T1]{fontenc}
\usepackage{pinlabel}
\graphicspath{{inked_pictures/}}

\newcommand{\dmo}{\DeclareMathOperator}
\dmo{\im}{im} %image
\dmo{\id}{id} %identity
\dmo{\Fun}{Fun} %functor
\dmo{\card}{card} %cardinality
\dmo{\res}{res} %restriction
\dmo{\Hom}{Hom} %homomorphism
\dmo{\SO}{SO} %special orthogonal group
\dmo{\ckr}{coker} %cokernel
\dmo{\Frob}{Frob} %Frobenius
\dmo{\GL}{GL} %general linear group
\dmo{\Gal}{Gal} %Galois group
\dmo{\ord}{ord} %order
\dmo{\Ab}{Ab} %abelian groups
\dmo{\Mor}{Mor} %morphisms
\dmo{\Ob}{Ob} %objects
\dmo{\End}{End} %endomorphisms
\dmo{\Nil}{Nil} %nilradical
\dmo{\Spec}{Spec} %spectrum
\dmo{\Tor}{Tor} %tor
\dmo{\ini}{in} %initial
\dmo{\disc}{Disc} %discriminant
\dmo{\Conf}{Conf}
\dmo{\UConf}{UConf}
\dmo{\Tup}{Tup}
\dmo{\Ver}{Ver}
\dmo{\Hor}{Hor}

\theoremstyle{plain}
\newtheorem{thm}{Theorem}[section]
\newtheorem{prop}[thm]{Proposition}
\newtheorem{lemma}[thm]{Lemma}
\newtheorem{cor}[thm]{Corollary}

\theoremstyle{definition}
\newtheorem{defin}[thm]{Definition}
\newtheorem{example}[thm]{Example}
\newtheorem{def/ex}[thm]{Definition/Example}

\newtheorem*{rmk}{Remark}
\newtheorem*{ack}{Acknowledgement}

\newcommand{\NN}{\mathbb{N}}

\newcommand{\RR}{\mathbb{R}}

\newcommand{\B}{\cal{B}}

\newcommand{\E}{\cal{E}}
\newcommand{\F}{\cal{F}}

\newcommand{\cal}[1]{\mathcal{#1}}

\newcommand{\sm}{\setminus}
\newcommand{\es}{\varnothing}

\newcommand{\al}{\alpha}
\newcommand{\be}{\beta}
\newcommand{\ga}{\gamma}

\newcommand{\ep}{\varepsilon}
\newcommand{\ze}{\zeta}
\newcommand{\si}{\sigma}
\newcommand{\Si}{\Sigma}
\newcommand{\de}{\delta}

\newcommand{\Th}{\Theta}

\newcommand{\sbs}{\subset}
\newcommand{\ti}{\times}
\newcommand{\ra}{\rightarrow}

\newcommand{\pa}{\partial}
\newcommand{\lt}{\leadsto}

\newcommand{\hra}{\hookrightarrow}

\newcommand{\til}{\tilde}

\title{Section Problems for Graph Configuration Spaces}
\author{Alexander Bauman}
 
\begin{document}

\maketitle

\begin{abstract}
We consider, for a finite graph $G$, when the surjective map $\Conf_{n+1}(G) \ra \Conf_n(G)$ of configuration spaces admits a section. We study when the answer depends only on the homotopy type of $G$, and give a complete answer. We also provide basic techniques for construction of sections.
\end{abstract}

\section{Introduction}

By a \textit{graph}, we mean a finite connected 1-dimensional cell complex. If $n$ is a natural number, then define the \textit{$n$-th ordered configuration space of $G$} to be the set
\[\Conf_n(G) = \{(x_1, \ldots, x_n)\in G^n : x_i \neq x_j \text{ for } i \neq j\}\]
equipped with the subspace topology in the product space $G^n$. Throughout, fix a graph $G$ and a natural number $n$.

Consider the ``forgetting map'' $\phi_{n+1}:\Conf_{n+1}(G) \ra \Conf_n(G)$ defined by:
\[(x_1, \ldots, x_{n+1}) \mapsto (x_1, \ldots, x_n).\]

Our main question is whether or not this surjection is split in the category of topological spaces, that is, whether or not there exists some map
\[s_n: \Conf_n(G) \ra \Conf_{n+1}(G)\]
such that
\[\phi_{n+1} \circ s_n = \id_{\Conf_n(G)}.\]
The map $s_n$ is called a \textit{section}. We will note that the data of a section is precisely the data of a map that continuously adds a point to each configuration of $\Conf_n(G)$.

We now state our main results, in terms of the Euler characteristic $\chi$ of graphs

\begin{thm}\label{mainthm}
For any graph $G$ with $\chi(G) < 0$ and natural number $n$ with $n \geq 2 - \chi(G)$, no section $s_n$ exists.
\end{thm}

\begin{thm}\label{chi0thm}
For any graph $G$ with $\chi(G) = 0$, a section $s_n$ exists for all natural numbers $n$.
\end{thm}

\begin{thm}\label{treesthm}
For any graph $G$ with $\chi(G) = 1$, a section $s_n$ exists if and only if $n \geq 2$.
\end{thm}

We also show that when $\chi(G) < 0$ and $n < 2 - \chi(G)$, then a section may or may not exist by providing examples of both cases, which shows that it isn't possible to improve our results without appealing to more than the value of $\chi(G)$.

The related problem of finding a section to the forgetting map $\Conf_{n+1}(M) \ra \Conf_n(M)$ where $M$ is a manifold with dimension greater than 1 has been studied, for instance in \cite{Fadell_Neuwirth_1962}, \cite{GONCALVES200333}, \cite{MR4052223}, and \cite{chen2019section}. This problem is somewhat more amenable to sophisticated arguments from algebraic topology since the forgetting map is a fibration, and the homology of $\Conf_n(M)$ is significantly better understood than the homology of $\Conf_n(G)$. The forgetting map $\phi_{n+1}$ of graph configuration spaces is not a fibration, as the topology of $G \sm \{*\}$ depends on our choice of point $*$. As a result, our methods are primarily geometric, and make little use of algebraic topology.

Another related problem is to determine whether the map $\phi_{n+1}$ induces a split surjection on fundamental groups. Since $\Conf_n(G)$ is aspherical, as proven in Theorem 3.1 of \cite{MR1873106}, this is equivalent to finding a map $s_n: \Conf_n(G) \ra \Conf_{n+1}(G)$ such that $\phi_{n+1} \circ s_n \simeq \id_{\Conf_n(G)}$. Proposition 5.6 of \cite{complexity} shows that such maps always exist.

The paper proceeds as follows: In Section 2, we introduce notation and terminology. In Section 3, we present elementary examples of sections. In Section 4, we introduce and develop the idea of \textit{connectable components}, a technical tool useful for the proof of Theorem \ref{mainthm}. Section 5 is devoted to the proof of Theorem \ref{mainthm}. In Section 6, we introduce a combinatorial model for $\Conf_n(G)$ which is useful for proving more advanced existence results. In Section 7, we show that nothing can be concluded about the case where $\chi(G) < 0$ and $n < 2 - \chi(G)$ without appealing to more than the homotopy type of the graph.

\begin{ack}
The author would like to thank Bena Tshishiku for introducing him to graph configuration spaces, and for meeting with him regularly to discuss ideas and review drafts of this paper.
\end{ack}

\section{Notation}

We assume that each graph is equipped with its natural cellular structure such that when $G \not\cong S^1$, every vertex has degree $\neq 2$, and when $G \cong S^1$, then $G$ only has one vertex. A \textit{branched vertex} is a vertex with degree $\geq 3$, and a \textit{free vertex} is a vertex with degree $1$. An edge $e$ is a \textit{loop} if it contains a single vertex. If we fix a homeomorphism between each edge and $I = [0, 1]$ (or $S^1$ in the case that the edge is a loop) we obtain a metric $d$ on $G$ by taking these homeomorphisms to be local isometries and using the shortest path metric. 

A subset of a graph is called a \textit{subgraph} if it is closed and connected (this is not the standard definition, but it will prove more useful and flexible for us). A graph is a \textit{tree} if it is simply connected. A \textit{maximal subtree} of a graph is a subgraph which is a tree, and which contains every vertex. It is a basic fact of graph theory that every graph contains a maximal subtree. 

We call the elements of $\Conf_n(G)$ \textit{configurations}, and the entries in a configuration \textit{tokens}. The integers $1, \ldots, n$ are called \textit{indices}. If $p, q$ are distinct points in $G$ which lie on a common edge, then $(p, q)$ (resp. $[p, q]$) is the open (resp. closed) interval they bound. In the case, for example, where $G \cong S^1$ or $p, q$ are vertices which are endpoints of multiple edges, there is ambiguity in this definition, which we ignore, but the particular choice of edge will be clear from context in all cases we consider. If $x = (x_1, \ldots, x_n)$ is a configuration, then we abuse notation, and denote by $G \sm x$ the complement $G \sm \{x_1, \ldots, x_n\}$.

\section{Basic Examples of Sections}\label{basicexamples}

Here we present three elementary cases where sections always exist. First we discuss identifying functions:

Note that the data of a section is exactly the data of a function that continuously adds a token to each configuration of $\Conf_n(G)$. In other words, this is equivalent to asking if there exists a map $f: \Conf_n(G) \ra G$ such that for any $x = (x_1, \ldots, x_n) \in \Conf_n(G)$ and any index $j$, we have that $f(x) \neq x_j$. We call any function of this form an \textit{identifying function}. Note that identifying functions are in one-to-one correspondence with sections. For notational simplicity we will prefer to work with them henceforth.

\begin{prop}\label{n1sect}
Suppose $n = 1$. If $\chi(G) \leq 0$ then $G$ admits an identifying function.
\end{prop}

\begin{proof}
If $\chi(G) \leq 0$, then there exists an embedded circle $S^1 \sbs G$ and a retract $r: G \ra S^1$. Define $f: G \ra G$ by $f(x) = -r(x)$, the point on $S^1$ antipodal to $r(x)$. Then $f$ is an identifying function.
\end{proof}

\begin{prop}\label{treesprop1}
If $n \geq 2$ and $\chi(G) = 1$, then $G$ admits an identifying function.
\end{prop}

\begin{proof}
Note that $\chi(G) = 1$ implies that $G$ is a tree. Therefore there exists a unique embedded line segment between $x_1, x_2$ for each configuration $x = (x_1, \ldots, x_n)$. Let 
\[\de(x) = \min(\{d(x_1, x_k): k \neq 1\}\cup \{1\})/2.\]
Then let $f(x)$ be the point on this embedded line segment with distance $\de(x)$ away from $x_1$. Then $f$ is an identifying function.
\end{proof}

Proposition \ref{treesprop1} is one half of Theorem \ref{treesthm}. We will prove the other direction in Proposition \ref{trees2}. We conclude this section by proving Theorem \ref{chi0thm}.

\begin{proof}[Proof of Theorem \ref{chi0thm}]
The case $n = 1$ follows from Proposition \ref{n1sect}, so assume $n \geq 2$. Our approach will be to orient each edge, and then add a point near $x_1$ with respect to this orientation. There is a unique embedded circle $S^1 \sbs G$. Orient $S^1$, and orient each edge in the complement of $S^1$ so they are ``pointing towards'' $S^1$ (Note the condition $\chi(G) = 0$ implies that $G$ is homeomorphic to a circle with trees attached, where each tree is attached at a single point). This gives us a unique ``forwards direction'' at each point on $G$. Let 
\[\de(x) = \min\{d(x_1, x_k): k \neq 1\}/2,\]
and let $f(x)$ be the point a distance $\de(x)$ away from $x_1$ in the ``forward direction.'' Then $f$ is an identifying function.
\end{proof}

\section{Connectable Components and Chasing}

In this section we develop tools which can be used to prove the nonexistence of sections. Consider the configurations $x, x'$ on the graph $G$, as depicted in Figure \ref{fig:changing_configs}.

$\,$

\begin{figure}[h]
    \labellist
    \small \hair 2pt
    \pinlabel $x_3$ at -15 55
    \pinlabel $f(x)$ at 50 125
    \pinlabel $x_1$ at 90 55
    \pinlabel $x_2$ at 210 60
    \pinlabel $x'_3$ at 270 55
    \pinlabel $x'_1$ at 390 125
    \pinlabel $x'_2$ at 500 60
    \pinlabel $Y_1$ at 330 115
    \pinlabel $Y_2$ at 460 115
    \pinlabel $Y_3$ at 380 20
    \endlabellist
    \centering
    \includegraphics[scale = .6]{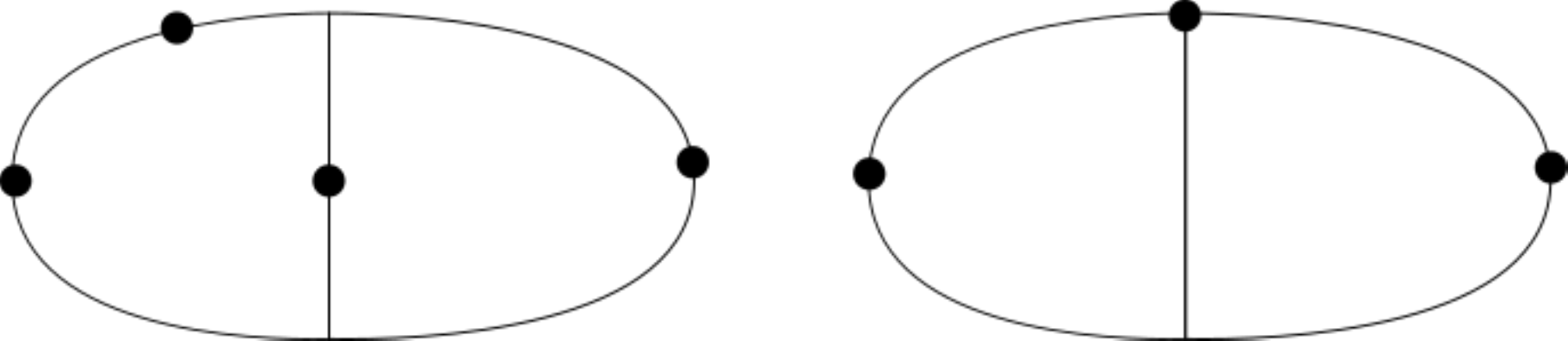}
    \caption{The configurations $x, x'$}
    \label{fig:changing_configs}
\end{figure}
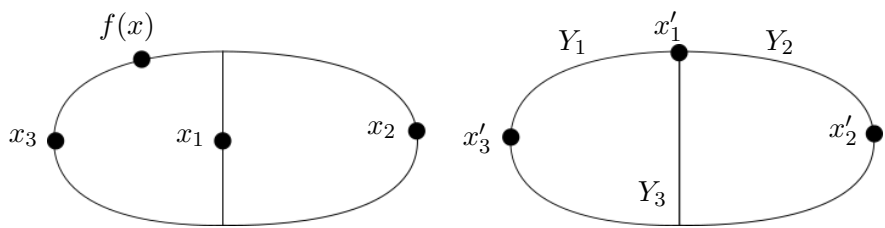

As we can see, $G = \Th$, the graph homeomorphic to a theta, $n = 3$, $x$ and $x'$ are the pictured configurations, and $Y_1, Y_2, Y_3$ are the connected components of $G \sm x'$. Suppose that $f$ is an identifying function on $G$, with $f(x)$ as shown on the left. We are interested in knowing which of $Y_1, Y_2, Y_3$ contains $f(x')$. Intuitively, we could imagine that if we move $x_1$ up along the central edge then we ``force'' $f(x)$ to choose between either $Y_1$ or $Y_2$, and that it is impossible for $f(x')$ to lie in $Y_3$.

Now consider the configuration $x''$ from Figure \ref{fig:type1move}

$\,$

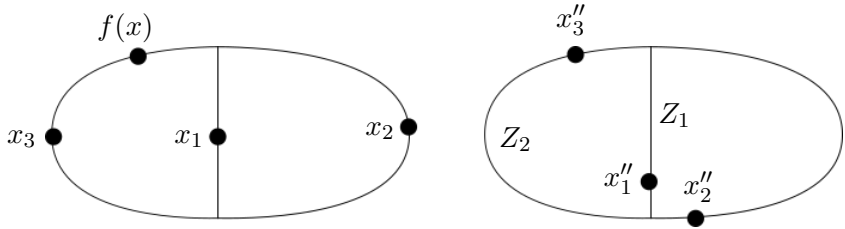
\begin{figure}[h]
    \labellist
    \small \hair 2pt
    \pinlabel $x_3$ at -15 55
    \pinlabel $f(x)$ at 50 125
    \pinlabel $x_1$ at 90 55
    \pinlabel $x_2$ at 210 60
    \pinlabel $x''_3$ at 330 130
    \pinlabel $x''_1$ at 360 30
    \pinlabel $x''_2$ at 410 25
    \pinlabel $Z_1$ at 395 70
    \pinlabel $Z_2$ at 295 55
    \endlabellist
    \centering
    \includegraphics[scale = .6]{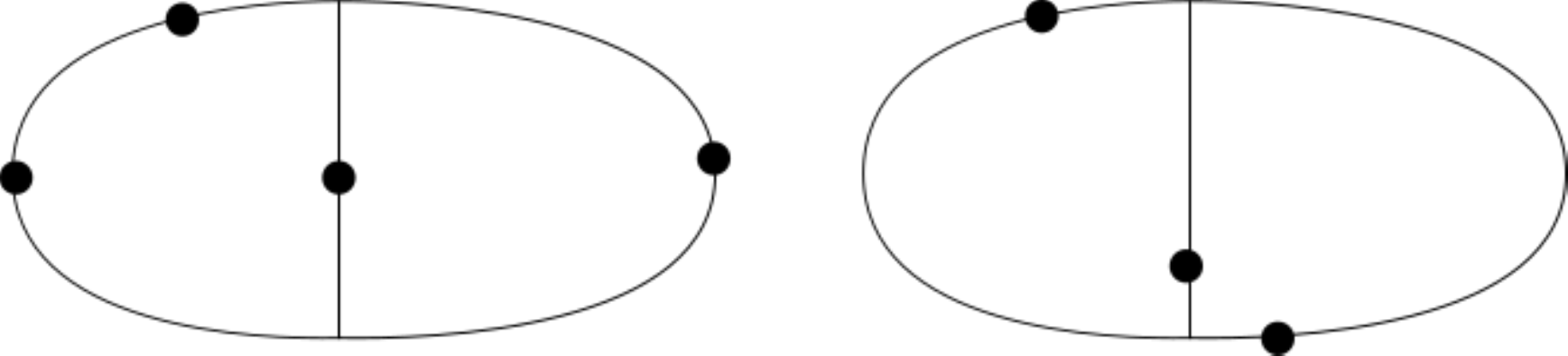}    
    \caption{The configuration $x''$}
    \label{fig:type1move}
\end{figure}

Note that $x''$ differs from $x$ only by moving tokens within edges, without ever entering vertices. Therefore, we would intuitively expect that $f(x'') \in Z_1$. These facts are not difficult to prove individually using basic concepts from point-set topology, but we would like generalizations of them that hold for every graph and any $n$. The main idea is that tokens entering branched vertices ``slice up'' the complement into more components, and along paths where no token enters a branched vertex, the components of the complement do not increase in number.
\subsection{Connectable Components}
Suppose $\al = (\al_1, \ldots, \al_n)$ is a path in $\Conf_n(G)$. We will define a relation between the components of $G \sm \al(0)$ and the components of $G \sm \al(1)$. First, we must define locally consistent systems of components along $\al$.

\begin{defin}
Suppose $\al = (\al_1, \ldots, \al_n)$ is a path in $\Conf_n(G)$, and suppose that for each $t \in I$ we have a connected component $X_t$ of $G \sm \al(t)$. The assignment $t \mapsto X_t$ is called a \textit{system of components along $\al$}. Such a system is \textit{locally consistent} if for each $t \in I$ and each point $p \in X_t$, there exists an $\ep > 0$ such that $p \in X_s$ for all $s \in (t - \ep, t + \ep)$.
\end{defin}

We now define the relation on the components of $G \sm \al(0)$ and the components of $G \sm \al(1)$ where $\al$ is a path in $\Conf_n(G)$.

\begin{defin}
Let $\al$ be a path in $\Conf_n(G)$. If $X$ is a component of $G \sm \al(0)$ and $Y$ is a component of $G \sm \al(1)$, then we say that $X$ and $Y$ are \textit{connectable by $\al$}, or write $X \lt_\al Y$ if there exists a locally consistent system $X_t$ of components along $\al$ such that $X_0 = X$ and $X_1 = Y$. We say that the system $X_t$ \textit{connects $X$ to $Y$} (along $\al$).
\end{defin}

We also must consider globally consistent systems of components defined on subsets of $\Conf_n(G)$.

\begin{defin}
Suppose $A \sbs \Conf_n(G)$. A \textit{(globally) consistent system of components on $A$} consists of a connected component $X_x$ of $G \sm x$, for each $x \in A$, such that for any path $\al: I \ra A$, the assignment $t \mapsto X_{\al(t)}$ is a locally consistent system of components along $\al$.
\end{defin}

The next proposition assures us of the essential fact that identifying functions give rise to consistent systems of components.

\begin{prop}\label{idfnlcs}
Suppose $f$ is an identifying function. Let $X_x$ be the component of $G \sm x$ containing $f(x)$. Then $x \mapsto X_x$ is a consistent system of components on $\Conf_n(G)$.
\end{prop}

\begin{proof}
Let $\al = (\al_1, \ldots, \al_n)$ be path in $\Conf_n(G)$, and let $X_t = X_{\al(t)}$. It suffices to show that the assignment $t \mapsto X_t$ is a locally consistent choice of components along $\al$. Select $t \in I$ and $p \in X_t$. Let $K$ be a line segment in $X_t$ connecting $p$ to $f(\al(t))$, and let $d(\al(t), K) = \de > 0$. Let $\phi: I \ra \RR$ be the continuous function defined by \[\phi(s) = d(\al(s), K \cup f(\al([s, t])))\] (where we instead write $[t, s]$ when $t < s$). Since $\phi(t) = \de > 0$, $\phi$ is positive in some neighborhood $(t - \ep, t + \ep)$ of $t$. For any $s \in (t - \ep, t + \ep)$, we can see that the connected set $K \cup f(\al([s, t]))$ contains $p$ and $f(\al(s))$, and is disjoint from $\al(s)$. Therefore, it lies in a single component of $G \sm \al(s)$, which is $X_s$, as desired.
\end{proof}

\begin{rmk}
Our main method of disproving the existence of sections is to show that consistent systems of components do not exist on $\Conf_n(G)$.

In this section, we will prove two important facts that will allow us to compute the relation $\lt_\al$ for two classes of paths. These facts formalize the intuitive ideas discussed in the opening portion of this section. Given an identifying function $f$, these facts will allow us to determine by Proposition \ref{idfnlcs} the component of $G \sm x$ containing $f(x)$ by moving along paths in $\Conf_n(G)$.
\end{rmk}

\begin{defin}
We say that a path $\al: I \ra \Conf_n(G)$ is a \textit{Type I path} if no token ever \textit{enters} a vertex. That is, if $\al_i(t)$ is not a vertex, then for all $s > t$, $\al_i(s)$ is also not a vertex. We say that a path is a \textit{Type II path} if each coordinate of $\al$ except one is constant, and the nonconstant coordinate moves directly at constant speed from the interior of an edge to a vertex. A path that is either Type I or Type II is called \textit{basic}.
\end{defin}

\begin{rmk}
If $\al$ is any path in $\Conf_n(G)$ such that at least two tokens never enter a branched vertex simultaneously, then $\al$ is homotopic to a concatenation of basic paths via a homotopy $h_t$ such that the topology of $G \sm h_t(s)$ depends only on $s$, up to reparameterization. If two tokens simultaneously enter branched vertices, then we can slightly perturb $\al$ such that the tokens enter the branched vertices at different times, which does not change the relation $\lt_\al$. If $\al, \be$ are paths such that the concatenation $\al * \be$ is defined, and $X$ is a component of $G \sm \al(0)$ and $Z $ is a component of $\be(1)$, then it is clear that $X \lt_{\al * \be} Z$ if and only if there is some component $Y$ of $G \sm \al(1)$ such that $X \lt_\al Y$ and $Y \lt_\be Z$. Therefore, if we compute the relation $\lt_\al$ for basic paths $\al$, we can determine it for all paths.
\end{rmk}

Our current goal is to determine exactly which components are connectable by basic paths.

\subsection{Type I Paths}

\begin{prop}\label{typeiuniqueness}
If $\al$ is a Type I path, and $X$ is a component of $G \sm \al(0)$, then there is at most one component $Y$ of $G \sm \al(1)$ such that $X \lt_\al Y$.
\end{prop}

\begin{proof}
Suppose $Y, Z$ are distinct components of $G \sm \al(1)$, and that $X \lt_\al Y$ and $X \lt_\al Z$. Suppose that $Y_t, Z_t$ are locally consistent systems of components along $\al$ connecting $X$ to $Y, Z$ respectively. Let
\[s = \inf \{t \in I : Y_t \neq Z_t\}.\]

If $Y_s = Z_s$, then for any $p \in Y_t$, we can take arbitrarily small $\ep > 0$ such that $Y_{s + \ep} \neq Z_{s + \ep}$ and $p \in Y_{s + \ep}$. However, this implies that $p \notin Z_{s + \ep}$ for arbitrarily small $\ep$, which contradicts the fact that $Z_t$ is a locally consistent system of components. Therefore, $Y_s \neq Z_s$.

Now, select $p \in Y_s, q \in Z_s$. There exists an $\ep > 0$ such that for all $r \in [s - \ep, s)$, we have $p, q \in Y_r$. Let $K$ be an embedded line segment in $Y_{s - \ep}$ connecting $p$ to $q$. Since $Y_s \neq Z_s$, there is some $r \in (s - \ep, s]$ such that for some $i = 1, \ldots, n$ we have $\al_i(r) \in K$. We can in fact assume that $\al_i(r) \in \pa K$ (where $\pa$ refers to the topological boundary as a subspace of $G$) by possibly taking a smaller value of $r$. However, by our choice of $\ep$, $\al_i(r) \neq p, q$, so $\al_i(r) = v$ for some branched vertex contained in $K$, since these are the only other boundary points of $K$. This contradicts the fact that $\al$ is a Type I path, and we are done.
\end{proof}

We now describe explicitly, for any type I path $\al$ and any component $X$ of $G \sm \al(0)$, the unique component $Y$ of $G \sm \al(1)$ such that $X \lt_\al Y$.

\begin{prop}\label{principle1}
Suppose $\al = (\al_1, \ldots, \al_n)$ is a Type I path, and that $X$ is a component of $G \sm \al(0)$. Define a component $Y$ of $G \sm \al(1)$ as follows:
\begin{enumerate}[(i)]
    \item If $X$ contains a vertex $v$, then $Y$ is the component containing $v$.
    
\begin{figure}[h]
    \labellist
    \hair 4pt
    \pinlabel $X$ at 80 180
    \endlabellist
    \centering
    \includegraphics[scale = .3]{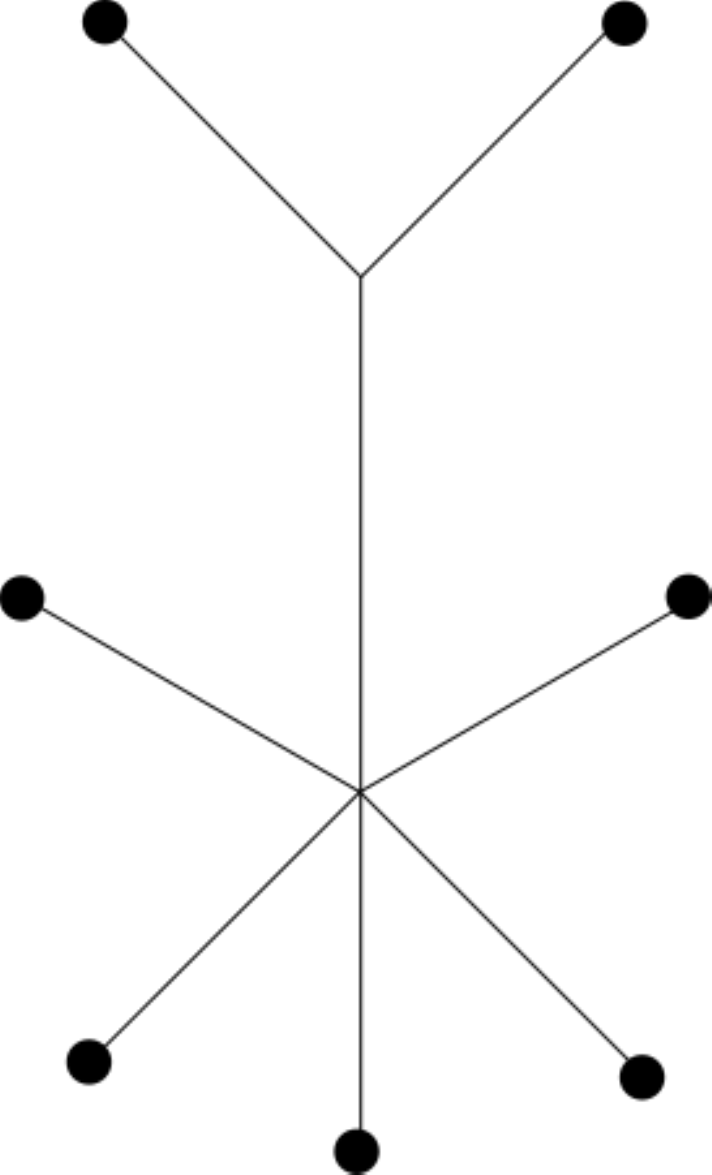}
    \caption{The case (i)}
\end{figure}
    
    \item If $X$ is an open interval $(\al_i(0), \al_j(0))$ in an edge $e$ and $\al_i(t), \al_j(t) \in e$ for all $t$, then $Y = (\al_i(1), \al_j(1)) \sbs e$.
    
\begin{figure}[h]
    \labellist
    \small \hair 2pt
    \pinlabel $x_i$ at 15 35
    \pinlabel $X$ at 150 40
    \pinlabel $x_j$ at 300 35
    \endlabellist
    \centering
    \includegraphics[scale  = .5]{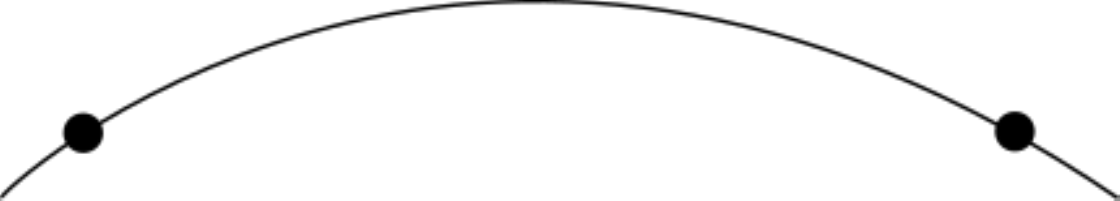}
    \caption{The case (ii)}
\end{figure}

    \item If $X$ is an open interval $(\al_i(0), \al_j(0))$ in an edge $e$ and $\al_i(0) = v$ is an endpoint of $e$, and $\al_i(t) \notin e$ for some $t$, then $Y$ is the component containing $v$.
    
\begin{figure}[h]
    \labellist
    \small\hair 2pt
    \pinlabel $x_j$ at -10 90 
    \pinlabel $X$ at 100 120
    \pinlabel $x_i$ at 235 105
    \endlabellist
    \centering
    \includegraphics[scale = .4]{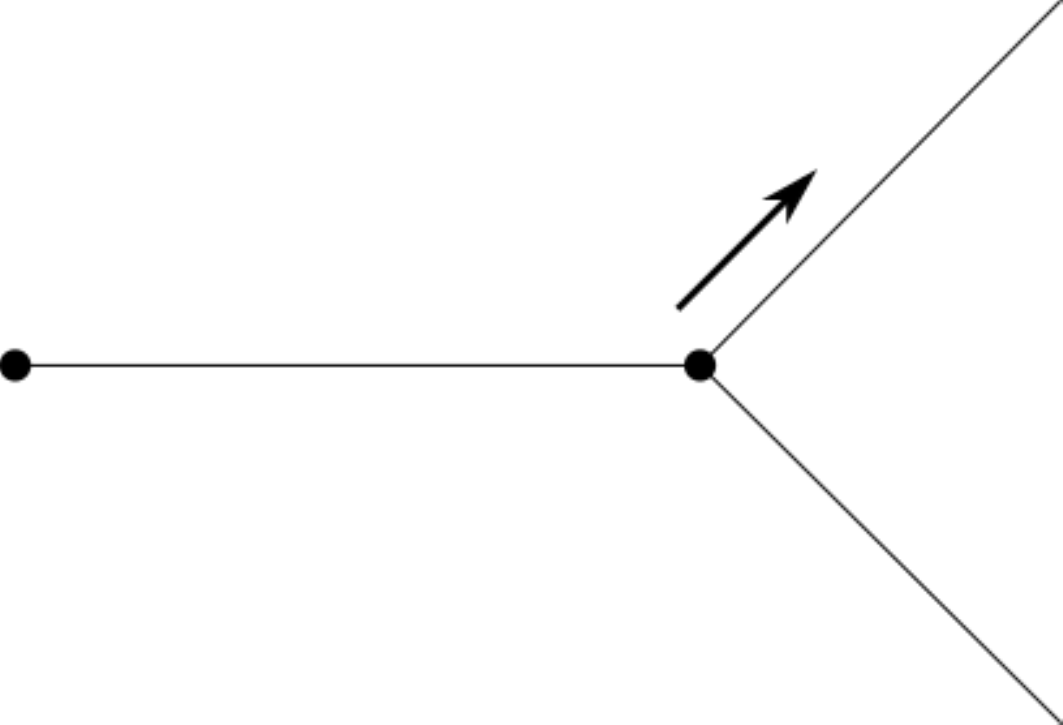}
    \caption{The case (iii)}
\end{figure}

    \item If $X$ is a loop $e$ without its branched vertex $v = \al_i(0)$, then if $\al_i$ is constant, $Y = X$, and otherwise, $Y$ is the component containing $v$.
    
    $\,$

\begin{figure}[h]
    \labellist
    \small\hair 2pt
    \pinlabel $x_i$ at 90 110
    \pinlabel $X$ at 90 230
    \endlabellist
    \centering
    \includegraphics[scale = .4]{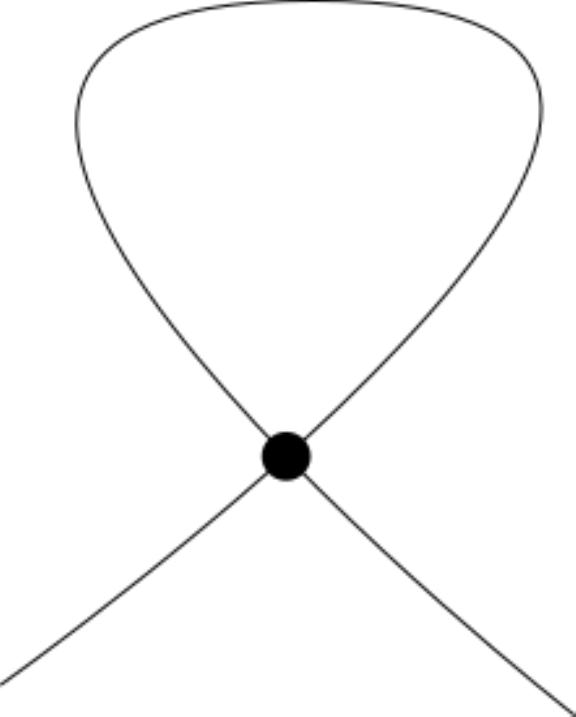}
    \caption{The case (iv)}
\end{figure}
\end{enumerate}
Then, $X \lt_\al Y$. 
\end{prop}

\begin{proof}

For any $t \in I$, consider the restricted path $\al|_{[0, t]}$. Since any restriction of a Type I path remains Type I, the statement of the proposition gives us a component $Y_t$ of $G \sm \al(t)$. We claim this is a locally consistent system of components along $\al$. We proceed by proving each of the cases separately.

\begin{enumerate}[(i)]
    \item Select $t \in I$ and $p \in Y_t$. Since $\al$ is Type I, $v \in Y_t$, so there exists a compact connected $K \sbs Y_t$ containing $p$ and $v$. Let $\ze = d(K, \al(t))$, and select $\ep$ such that for all $s \in (t - \ep, t + \ep)$ and all $i = 1, 2, \ldots, n$, we have $d(\al_i(t), \al_i(s)) < \ze$. Since $v \in K$, we have $K \sbs Y_s$ for all $s \in (t - \ep, t + \ep)$, so $p \in Y_s$, as desired.
    
    \item Select $t \in I$ and $p \in Y_t$. Let \[\ze = \min\{d(\al_i(t), p), d(\al_j(t), p)\},\] and select $\ep$ such that $d(\al_\eta(t), \al_\eta(s)) < \ze$ for $\eta = i, j$ and $s \in (t - \ep, t + \ep)$. Then $p \in (\al_i(s), \al_j(s)) = Y_s$ for all $s \in (t - \ep, t + \ep)$, as desired.
    
    \item Select $t \in I$. If $\al_i(t + \ep) = v$ for some $\ep > 0$, then apply (ii) to the restricted path $\al|_{[0, t + \ep]}$. If $\al_i(t - \ep) \notin e$ for some $\ep > 0$, then apply (i) to the restricted path $\al|_{[t - \ep, 1]}$. It suffices to consider the remaining case, where $\al_i(t) = v$ but $\al_i(t + \ep) \notin e$ for all $\ep > 0$. Select $p \in Y_t$. As before, we can set \[\ze = \min\{d(\al_i(t), p), d(\al_j(t), p)\},\] and select $\ep$ such that for all $s \in (t - \ep, t + \ep)$ we have $d(\al_\eta(s), \al_\eta(t)) < \ze$ for $\eta = i, j$. If $s \leq t$, then $Y_s = (\al_i(s), \al_j(s))$, which contains $p$. If $s > t$, then $Y_s$ contains $v$. Since $\al_i(s) \notin e$, $[v, p]$ is disjoint from $\al(s)$, so $p \in Y_s$.
    
    \item We can modify the cellular structure of $G$ such that $G$ has no loops, perhaps by adding a vertex to each loop. Since none of our techniques depend on the vertices of $G$ being branched, we can reduce to part (i).
\end{enumerate}
\end{proof}

\subsection{Type II Paths}
We now turn to Type II paths. We begin with a lemma.

\begin{lemma}\label{growingcomponents}
Suppose $\al$ is a path in $\Conf_n(G)$, and $X_t$ is a system of components along $\al$. If $X_t \sbs X_s$ whenever $t < s$, then $X_t$ is the \textit{only} locally consistent system of components along $\al$ starting at $X_0$.
\end{lemma}

\begin{proof}
We can easily observe that $X_t$ is locally consistent, so it suffices to show that this is the unique locally consistent system of components along $\al$ starting at $X_0$. Suppose $Y_t$ is another such system, and let
\[s = \inf \{t \in I : X_t \neq Y_t\}.\]
By the argument from the proof of Proposition \ref{typeiuniqueness}, we can see that in fact $X_s \neq Y_s$. Select $p \in Y_s$, and $\ep$ such that $p \in Y_{r}$ for all $r \in [s - \ep, s]$. Since $Y_r = X_r$ by assumption, and $X_r \sbs X_s$, we have $p \in X_s$, a contradiction.
\end{proof}

We can now compute the relation $\lt_\al$ when $\al$ is a Type II path.

\begin{prop}\label{principle2}
Let $\al = (\al_1, \ldots, \al_n)$ be a Type II path where $\al_i$ is the unique nonconstant coordinate, and therefore has constant speed, and $\al_i(1) = v$ for some vertex $v$. Suppose that the components of $G \sm \al(0), G \sm \al(1)$ near $v$ are labelled as in the picture below (note that $X_1, X_2$ and $Y_1, \ldots, Y_k$ are not necessarily distinct). If $X = X_1$, and $X_1, X_2$ are distinct, then $Y = Y_1$, and if $X = X_2$ then $Y = Y_j$ for some $j \geq 2$ . If $X = Z$ for some other component $Z$, then $Y = Z$.

\begin{figure}[h]
    \labellist
    \small\hair 2pt
    \pinlabel $X_2$ at 45 220
    \pinlabel $X_1$ at 45 80
    \pinlabel $Y_1$ at 315 130
    \pinlabel $Y_2$ at 315 220
    \pinlabel $Y_3$ at 290 195
    \pinlabel $Y_4$ at 295 150
    \pinlabel $Y_5$ at 355 155
    \endlabellist
    \centering
    \includegraphics[scale = .7]{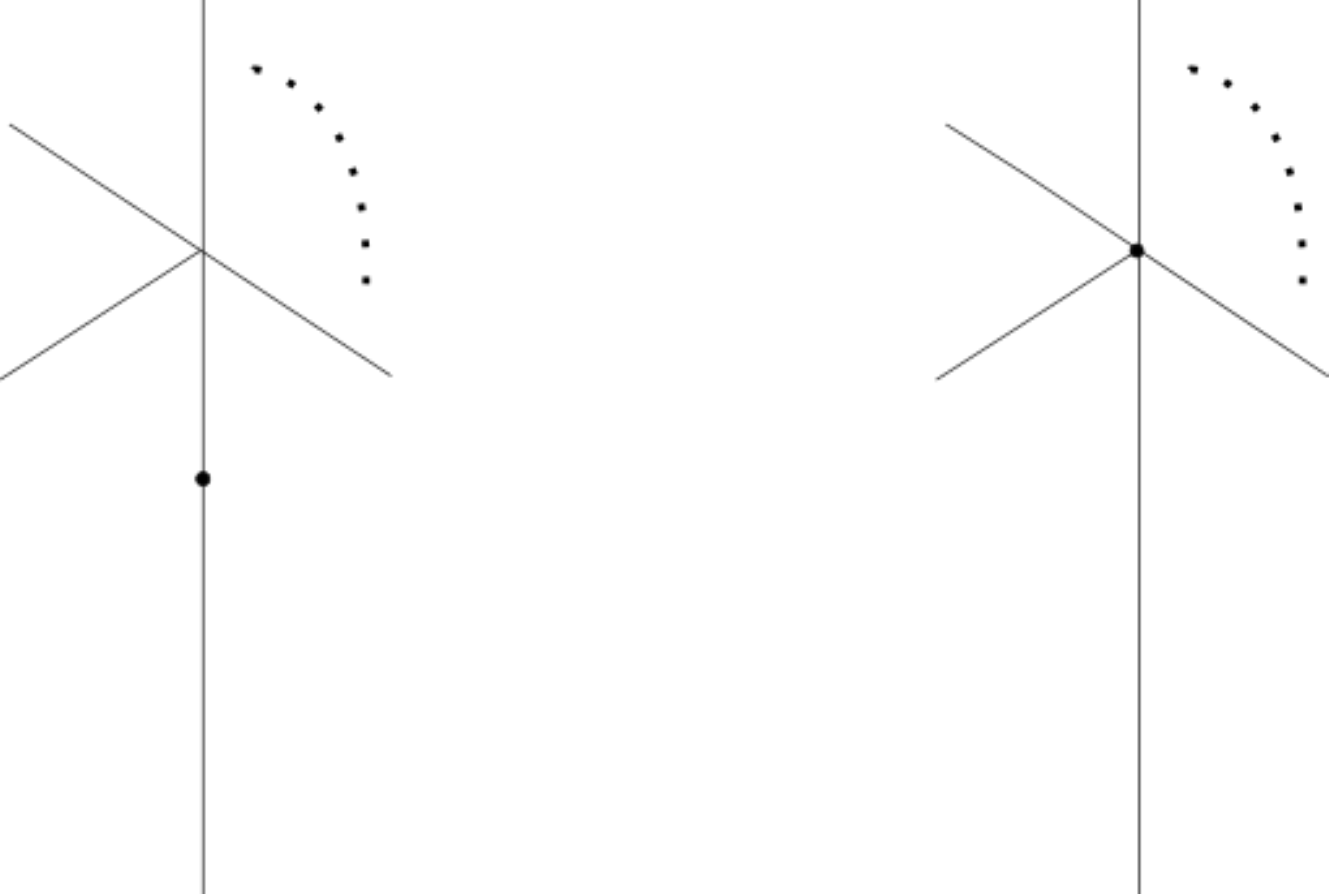}
\end{figure}
\end{prop}

\begin{proof}
First, suppose that $X = X_1$ and that $X_1, X_2$ are distinct. Then, we can define a system of components $Z_t$, such that $Z_t$ is the component of $G \sm \al(t)$ containing $X_1$. By Lemma \ref{growingcomponents}, this is the unique locally consistent system of components starting at $X_1$, so $Y = Y_1$.

Now suppose that $X = X_2$, and that $Z_t$ is a locally consistent system of components starting at $X_2$. It suffices to show that \[Z_1 \sbs Y_2 \cup \cdots \cup Y_k.\] Select $p \in Z_1$, and $\ep$ such that $p \in Z_s$ whenever $s \in [1-\ep, 1]$. For any $s < 1$, we note that the restriction $\al|_{[0, s]}$ is a Type I path, which implies that $Z_s$ is the component of $G \sm \al(s)$ containing $v$ by Proposition \ref{principle1}(i). When $Y_1 = Y_j$ for some $j \geq 2$, this implies that \[Z_s = (Y_2 \cup \cdots \cup Y_k) \sm \{\al_i(s)\},\]

and since $p \in Z_{1 - \ep}$, this implies that $p$ is in $Y_j$ for some $j \geq 2$, as desired. When $Y_1 \neq Y_j$ for all $j \geq 2$, then
\[Z_s = Y_2 \cup \cdots \cup Y_k \cup [\al_i(s), v].\]

If we can show $p \notin [\al_i(1-\ep), v]$, then we are done. However, this is impossible, since otherwise there would be some $s \in (1-\ep, 1)$ such that $\al_i(s) = p$, which contradicts our choice of $\ep$. This completes the proof.
\end{proof}

Propositions \ref{principle1} and \ref{principle2} are generally used in conjunction with Proposition \ref{idfnlcs}. If we have a basic path $\al$, and an identifying function $f$ with $f(\al(0)) \in X$ for some component $X$ of $G \sm \al(0)$, then we obtain by Proposition \ref{idfnlcs} a locally consistent choice of components along $\al$. We can then use either Proposition \ref{principle1} or Proposition \ref{principle2} to determine the possible components of $G \sm \al(1)$ which can contain $f(\al(1))$.

\subsection{Applications}

Our first important application of Propositions \ref{principle1} and \ref{principle2} is the ``chasing'' lemma:
\begin{lemma}[Chasing]\label{chasing}
Suppose $f$ is an identifying function, and $x = (x_1, \ldots, x_n)$ is a configuration. Let $X$ denote the connected component of $G \sm x$ containing $f(x)$, and suppose that $X \cup \{x_1\}$ is simply connected. Then there exists some point $p \in X$ and an integer $2 \leq j \leq n$ such that if $x' = (p, x_2, \ldots, x_n)$, then $f(x') \in (x_j, p)$. 
\end{lemma}

\begin{proof}
We induct on the number of vertices $k$ in $X$. If $k = 0$, then $X$ is an interval, so we can merely take $p = x_1$.

Suppose we know the result for all $\ell < k$. Since $X \cup \{x_1\}$ remains connected when $x_1$ is removed, $x_1$ must have degree 1 in the closure of $X$, since any point which is not a free vertex disconnects a simply connected 1-complex when removed. Consider the Type II path $\al$ in $\Conf_n(G)$ such that $\al_1(0) = x_1$ and $\al_1$ moves along the unique edge in the closure of $X$ which is connected to $x_1$. Let $v = \al_1(1)$, and let $X'$ be the connected component of $G \sm \al(1)$ containing $f(\al(1))$. By Proposition \ref{principle2}, $X'$ borders $v$ and $x_1 \notin X'$. Therefore, $X' \cup \{v\}$ is simply connected, since it is a connected subset of a tree. Since $v \notin X'$ and $X' \sbs X$, $X'$ contains fewer than $k$ vertices, so we are done by the inductive hypothesis.
\end{proof}

Of course, the chasing lemma applies equally well when $x_1$ is replaced by any other token. The chasing lemma, while simple, is an important aspect of the proof of the main theorem. When we apply it (with $x_i$ taking the role of $x_1$) we say that we are ``chasing with $x_i$.''

The following corollary is worth noting separately:

\begin{cor}\label{treechasing}
If $f$ is an identifying function on $G$, and $T \sbs G$ is a subtree of $G$ which is connected to the rest of $G$ by a single point $p$, then for any configuration of the form $x = (p, x_2, \ldots, x_n)$ with $x_2, \ldots, x_n \notin T$, we have $f(x) \notin T$.
\end{cor}

\begin{proof}
Suppose that $x = (p, x_2, \ldots, x_n)$ is such that $x_2, \ldots, x_n \notin T$, and $f(x) \in T$. Then, applying Lemma \ref{chasing}, there is some $i = 2, \ldots, n$ and $p' \in T$ such that if $x' = (p', x_2, \ldots, x_n)$, then $f(x') \in (p', x_i)$. If $p' \neq p$, then this is a contradiction since $x_i \notin T$, and if there is any interval where exactly one endpoint is in $T$, then that endpoint must be in $\pa T$, and therefore be equal to $p$. However, if $p' = p$, then $x' = x$, and we have $f(x) \notin T$, which is a contradiction, so we are done.
\end{proof}

We can now prove the second half of Theorem \ref{treesthm}.

\begin{prop}\label{trees2} Suppose $\chi(G) = 1$ and $n = 1$. Then $G$ does not admit a section.
\end{prop}

\begin{proof}
Suppose $f:G \ra G$ is an identifying function. Given any point $p \in G$, $p$ separates $G$ into subtrees $T_1, \ldots, T_k$. However, by Corollary \ref{treechasing}, $f(p)$ can be contained in none of them, which is a contradiction.
\end{proof}

Note that Proposition \ref{trees2} is a generalization of the 1-dimensional Brouwer fixed-point theorem, as it says that for any tree $T$ (not just the interval $I$), any continuous map $T \ra T$ has a fixed point.

\begin{rmk}
This discussion gives an algebraic necessity for the existence of a section. Let $PB_n(G) = \pi_1(\Conf_n(G))$ denote the \textit{$n$'th pure braid group} of $G$. Suppose that $X_x$ is a consistent system of components on $\Conf_n(G)$. Let $H < PB_{n+1}(G)$ be the set of homotopy classes of paths which can be represented by a path $\al = (\al_1, \ldots, \al_{n+1})$ such that $\al_{n+1}(t) \in X_{(\al_1(t), \ldots, \al_n(t)))}$ for all $t \in I$. One can verify that $H$ is a subgroup of $PB_{n+1}(G)$. Then, if $s: \Conf_n(G) \ra \Conf_{n+1}(G)$ is a section whose corresponding system of components (in the sense of Proposition \ref{idfnlcs}) is $X_x$, then the image of the induced map $s_*: PB_n(G) \ra PB_{n+1}(G)$ must lie in $H$, since for any based loop $\be$ in $\Conf_n(G)$, the composition $s \circ \be$ represents an element of $H$. Therefore, if a section exists then there exists some consistent choice of components $X_x$ such that the restricted map:
\[H \hra  PB_{n+1}(G) \ra PB_n(G)\]
admits a section.
\end{rmk}

\section{Proof of Theorem \ref{mainthm}}

In this section, we prove Theorem \ref{mainthm} and some necessary lemmas. 

\subsection{Distinguished Pairs}
First we define and discuss the basic properties of distinguished pairs, which occur when, for a configuration $x$ and identifying function $f$, $f(x)$ lies between two tokens of $x$. We will show that the existence of such pairs allow us to draw broad conclusions about the existence of sections, culminating in the proof of Theorem \ref{mainthm}. Recall the open interval $(p, q)$ notation introduced at the start of Section \ref{basicexamples}.

\begin{defin}\label{distpair}
Suppose $G \not\cong S^1$, and $e$ is an oriented edge such that $G$ remains connected when a single interior point of $e$ is removed, and suppose that $f$ is an identifying function on $G$. If $i, j$ are a pair of distinct indices, we say that $(i, j)$ forms a \textit{distinguished pair (of $f$) on $e$} if there exists a configuration $x = (x_1, \dots, x_n)$ such that:
\begin{enumerate}[(a)]
    \item $x_i, x_j \in e$,
    \item $x_i$ lies before $x_j$ on $e$ with respect to the orientation on $e$,
    \item $x_k \notin (x_i, x_j)$ for all $k$,
    \item $f(x) \in (x_i, x_j)$.
\end{enumerate}

We then say that the configuration $x$ is a \textit{witness} to the distinguished pair $(i, j)$.
\end{defin}

\begin{figure}[h]
    \labellist
    \small\hair 2pt
    \pinlabel $x_i$ at 10 30
    \pinlabel $f(x)$ at 150 40
    \pinlabel $x_j$ at 300 30
    \endlabellist
    \centering
    \includegraphics[scale = .7]{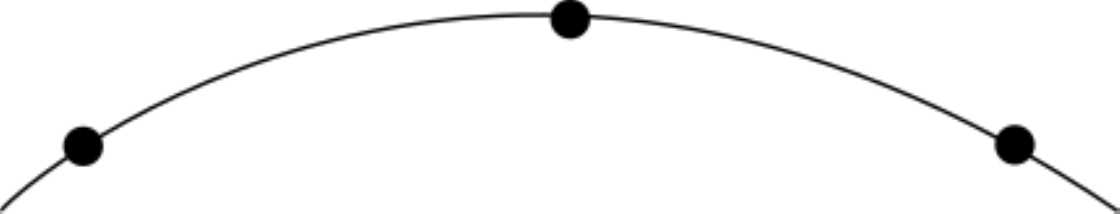}
    \caption{A witness to a distinguished pair $(i, j)$}
\end{figure}

Note that if $f$ is an identifying function and $x = (x_1, \ldots, x_n)$ is a configuration which satisfies the hypotheses of Lemma \ref{chasing} with respect to $f$, and additionally $G \sm \{x_\ell\}$ is connected for each $\ell \geq 2$, then the conclusion of Lemma \ref{chasing} in fact produces a distinguished pair containing $1$.

Next we state and prove the most important fact about distinguished pairs.

\begin{prop}\label{distpairprop}
If $(i, j)$ is a distinguished pair on $e$, and $y = (y_1, \ldots, y_n)$ is a configuration which satisfies the properties (a)-(c) in Definition \ref{distpair} with respect to $(i, j)$ for $f$, then $y$ also satisfies property (d), that is, it is a witness to $(i, j)$.
\end{prop}

\begin{proof}
Let $x = (x_1, \ldots, x_n)$ to $(i, j)$ be a witness to $(i, j)$, and let $y = (y_1, \ldots, y_2)$ be as in the proposition statement. Our goal is to construct a path $\al = (\al_1, \ldots, \al_n)$ in $\Conf_n(G)$ connecting $x, y$ such that $\al_i, \al_j$ never leave $e$, which is a concatenation of basic paths. This will imply by Propositions \ref{principle1} and \ref{principle2} that $f(y) \in (y_i, y_j)$.

First consider the case where none of the $x_\ell$ with $\ell \neq i, j$ lie in $e$, and $y_i, y_j$ lie in the interior of $e$. Then, we can take $\al$ to first move the $i$'th and $j$'th tokens to $y_i, y_j$ respectively, and then move each of the remaining $n-2$ tokens of $x$ to the corresponding tokens of $y$. Since $\Conf_{n-2}(G \sm [x_i, x_j])$ is connected by Theorem 2.7 of \cite{MR2701024}, this is possible.

If (say) $y_i$ is at an endpoint of $e$, then let $y_i'$ be a point in $e$ which is very close to $y_i$. We can modify the path $\al$ above by initially moving $x_i$ to $y_i'$, and then moving $y_i'$ to $y_i$ once the remaining $n-2$ tokens agree with $y$.

If some of the $x_\ell$ lie in $e$ for $\ell \neq i, j$, then we can first apply a path which moves each of these $x_\ell$ outside of $e$, and then reduce to the previous case.

All these paths can be taken to be concatenations of basic paths, so we are done.
\end{proof}

Proposition \ref{distpairprop} gives rise to the very useful observation:

\begin{cor}\label{distpaircor1}
If $(i, j)$ is a distinguished pair on $e$ and $(k, \ell)$ is a distinguished pair on $e'$, then $\{i, j\} \cap \{k, \ell\} \neq \es$
\end{cor}

\begin{proof}
Suppose $(i, j)$, $(k, \ell)$, are distinguished pairs such that $\{i, j\} \cap \{k, \ell\} = \es$. Then, we can find a configuration $x = (x_1, \ldots, x_n)$ satisfying the hypotheses for Proposition (\ref{distpairprop}) with respect to both distinguished pairs. This necessitates $f(x) \in (x_i, x_j)$ and $f(x) \in (x_k, x_\ell)$, which is a contradiction.
\end{proof}

\subsection{The Case $n > 3$}

We now present the proof of the main theorem in the case $n > 3$:

Suppose $f$ is an identifying function, and consider the statements below:

\begin{enumerate}[(A)]
    \item Every index belongs to some distinguished pair of $f$.
    
    \item No index belongs to every distinguished pair of $f$.
\end{enumerate}

\begin{lemma}
If $n > 3$ and $f: \Conf_n(G) \ra G$ is an identifying function, then the statements (A), (B) are not simultaneously true of $f$.
\end{lemma}

\begin{proof}

Suppose that statements (A), (B) are both true, and suppose without loss of generality that $(1, 2)$ is a distinguished pair. Then $3$ is contained in some distinguished pair, $(i, 3)$. If $i \neq 1, 2$, then this contradicts Corollary \ref{distpaircor1}, so suppose $i = 1$. Then, $(1, 4)$ must be a distinguished pair, since if $4$ is in any other distinguished pair, we can produce a counterexample to Corollary \ref{distpaircor1}. By (B), there is some distinguished pair $(i, j)$ with $i, j \neq 1$. This contradicts Corollary \ref{distpaircor1} applied to either the pairs $(i, j)$ and $(1, 2)$, the pairs $(i, j)$ and $(1, 3)$, or the pairs $(i, j)$ and $(1, 4)$, depending on $i, j$.
\end{proof}

It therefore suffices to establish that (A) and (B) hold when $n \geq 2 - \chi(G)$ and $\chi(G) < 0$.

\begin{proof}[Proof of (A) when $n \geq 2 - \chi(G)$]

Let $T$ be a maximal subtree of $G$, let $k = 1 - \chi(G)$, and let $e_1, e_2, \ldots, e_k$ be the edges in the complement of $T$. Let $A_1, A_2$ be closed intervals contained in the interiors of $e_1, e_2$ respectively, and let
\[\ga: I \ra T \cup e_1 \cup e_2\]
be an embedded path connecting an endpoint of $A_1$ to an endpoint of $A_2$, such that the intersection of the image of $\ga$ with $A_1 \cup A_2$ is exactly these two endpoints. We can assume that $T$ is not contained in the image of $\ga$ (if not, enlarge $T$ slightly so that it is not homeomorphic to an interval). The situation is summarized in Figure \ref{fig:exampleA}.

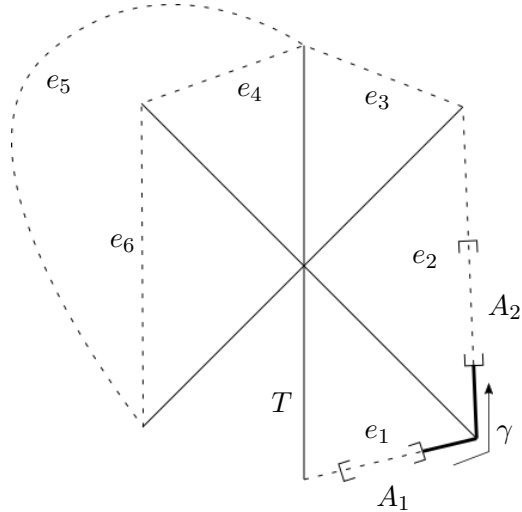
\begin{figure}[h]
    \labellist
    \small\hair 2pt
    \pinlabel $e_1$ at 230 30
    \pinlabel $e_2$ at 260 140
    \pinlabel $e_3$ at 230 240
    \pinlabel $e_4$ at 150 245
    \pinlabel $e_5$ at 30 250
    \pinlabel $e_6$ at 70 150
    \pinlabel $\ga$ at 310 30
    \pinlabel $A_1$ at 240 -10
    \pinlabel $A_2$ at 310 110
    \pinlabel $T$ at 170 50
    \endlabellist
    \centering
    \includegraphics[scale = .6]{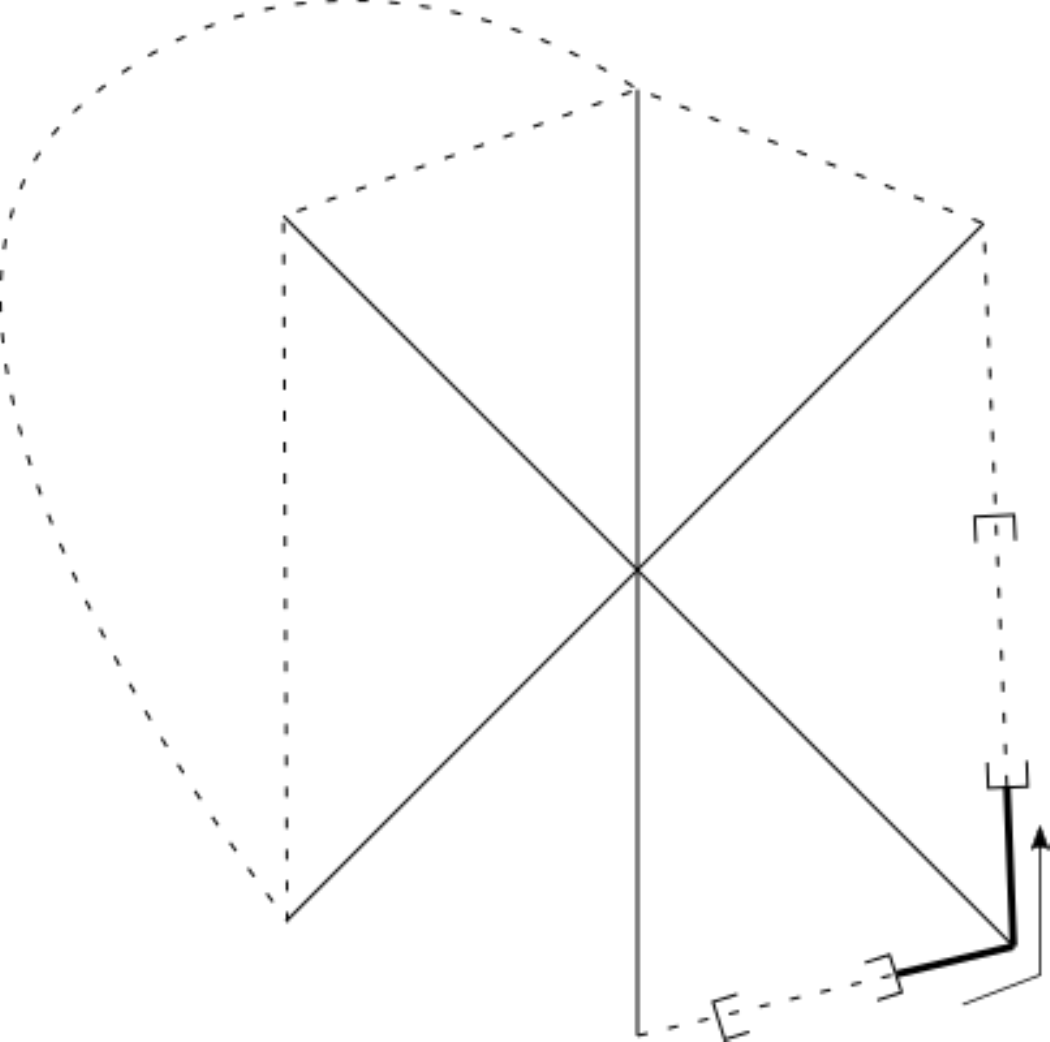}
    \caption{A graph $G$ with labelled edges, with labelled $A_1$, $A_2$ and $\ga$.}
    \label{fig:exampleA}
\end{figure}

We will prove that $1$ is contained in some distinguished pair. Let $x_1 \in G$ be a point in $T$ not contained in the image of $\ga$. Let $x_2, \ldots, x_{k-1}$ be points in the interiors of $e_3, \ldots, e_k$ respectively. Let $x_k, \ldots, x_{n-1}$ be points spaced out equally in $A_1$ in order such that $x_{n-1}$ is on the endpoint contained in the image of $\ga$, and let $x_n$ be the endpoint of $A_2$ not contained in the image of $\ga$. Let $x = (x_1, \ldots, x_n)$. The configuration $x$ is pictured in Figure \ref{fig:labelledA}.

\begin{figure}[h]
    \labellist
    \small \hair 2pt
    \pinlabel $x_1$ at 230 150
    \pinlabel $x_2$ at 240 310
    \pinlabel $x_3$ at 150 320
    \pinlabel $x_4$ at 20 300
    \pinlabel $x_5$ at 70 220
    \pinlabel $x_6$ at 200 20
    \pinlabel $x_7$ at 220 30
    \pinlabel $x_{n-1}$ at 340 40
    \pinlabel $x_n$ at 380 270
    \endlabellist
    \centering
    \includegraphics[scale = .7]{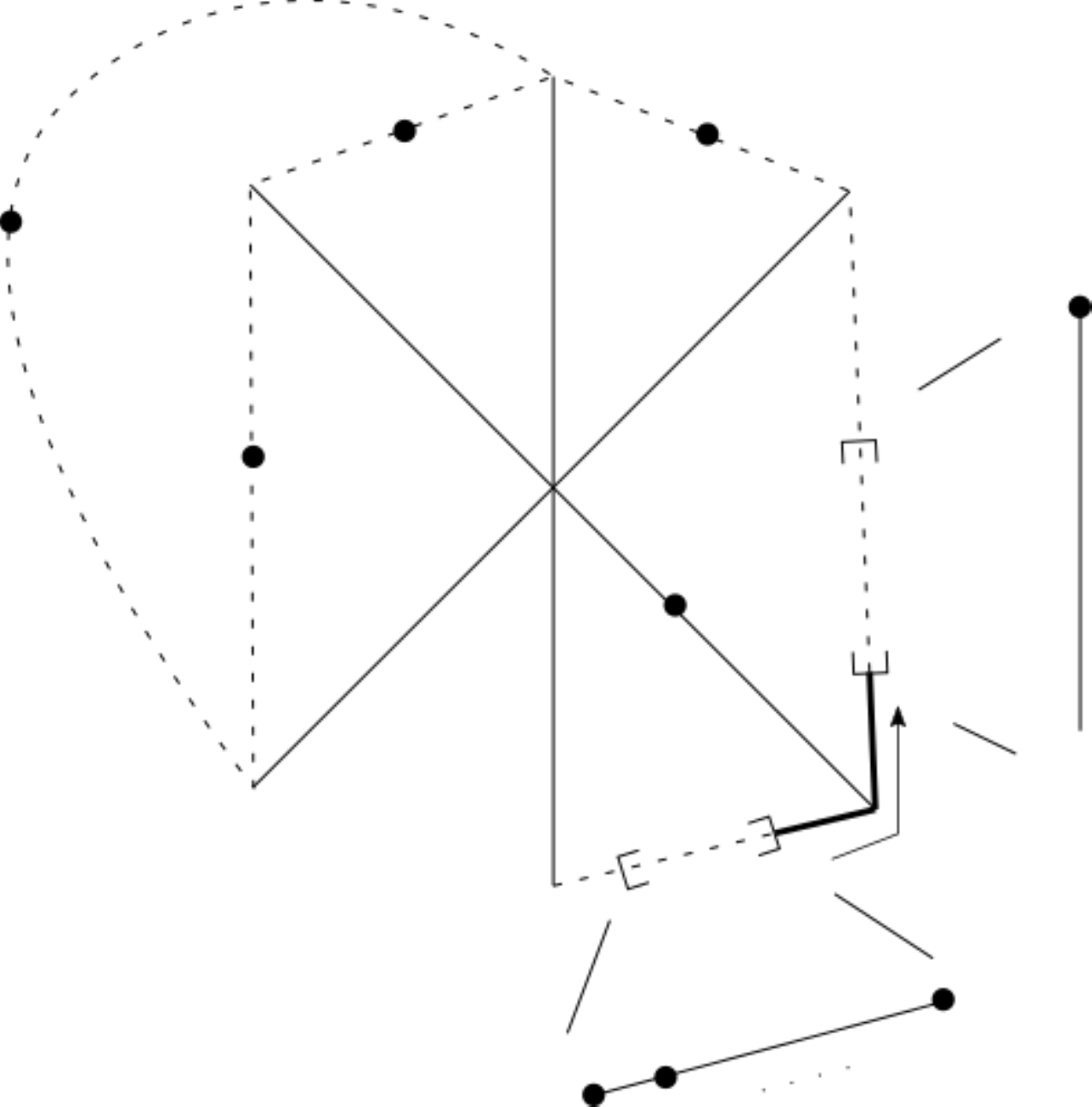}
    \caption{The configuration $x$}
    \label{fig:labelledA}
\end{figure}
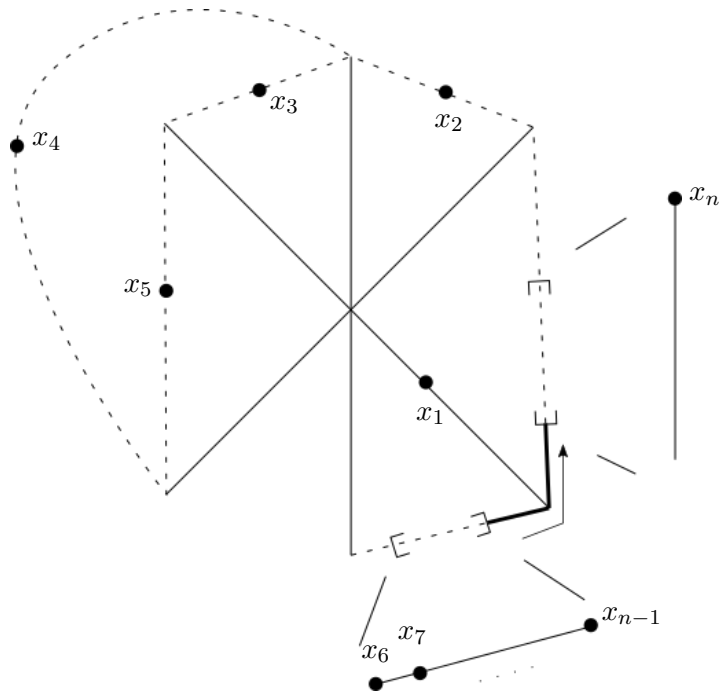

The connected components of $(G \sm x) \cup \{x_1\}$ are the intervals $(x_k, x_{k+1}), \ldots, (x_{n-2}, x_{n-1})$, and a large simply connected component $X$ which contains $x_1$. If $f(x) \in X$, then let $\til{X}$ be the component of $X \sm \{x_1\}$ containing $f(x)$. Note that $\til{X} \cup \{x_1\}$ is simply connected, so we can apply Lemma \ref{chasing} to produce a distinguished pair containing 1, as desired.
 
Otherwise, suppose that $f(x) \in (x_{n-2}, x_{n-1})$ (the cases when $f(x)$ lies in a different of these intervals is similar, by renumbering the indices and using Proposition \ref{distpairprop}). Let $\al$ be the path in $\Conf_n(G)$ which moves $x_{n-1}$ along $\ga$, and fixes each of the other tokens. Call the ending configuration
\[x' = (x_1, \ldots, x_{n-2}, \ga(1), x_n),\]
and let $X'$ be the big simply connected component in $(G \sm x') \cup \{x_1\}$. For each $t$, let $X_t$ be the component of $G \sm \be(t)$ which contains $(x_{n-2}, x_{n-1})$. Note that $X_t$ is a system of components which satisfies the hypotheses of Lemma \ref{growingcomponents}, which implies that $f(x') \in X_1 = X'$. As before, we can apply Lemma \ref{chasing} to $x'$ to produce a distinguished pair containing $1$.
\end{proof}

\begin{proof}[Proof of (B) when $n \geq 2 - \chi(G)$]

Suppose for sake of contradiction that $1$ is in every distinguished pair, and that $(1, 2)$ is a distinguished pair in $e_1$ (a distinguished pair exists by (A)). Recall that this equips $e_1$ with an specified orientation. Let $T$ be a maximal subtree not containing $e_1$, and let $e_2$ be some other edge in $G \sm T$. Suppose further that $T$ ``peeks in'' to each edge, such that each point in $\pa T$ is contained in the closure of exactly one edge in $G \sm T$. Figure \ref{fig:maxtree} show an example of one of these constructions.

\begin{figure}[h]
    \labellist
    \small\hair 2pt
    \pinlabel $e_1$ at -20 70
    \pinlabel $e_2$ at 170 80
    \pinlabel $T$ at 50 50
    \endlabellist
    \centering
    \includegraphics[scale = .6]{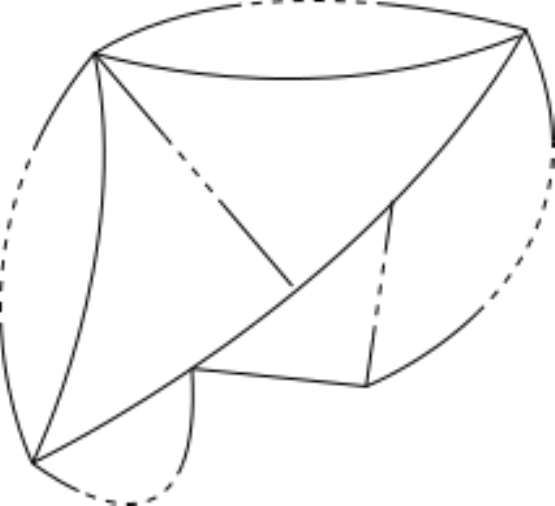}    
    \caption{}
    \label{fig:maxtree}
\end{figure}

Let $e_3, \ldots, e_k$ be the other edges not contained in $T$, where $k = 1 - \chi(G)$, as before. Let $H$ be the subgraph formed from $T \cup e_1 \cup e_2$ by removing free edges until every vertex has degree at least $3$. Since $H$ is a deformation retract of $T \cup e_1 \cup e_2$, we can see that $\chi(H) = -1$, and $H$ has no free edges. Therefore $H$ is homeomorphic to either the figure eight graph $(\infty)$, the theta graph $(\Th)$, or the dumbbell graph $D$. The three possibilities are pictured in Figure \ref{fig:subgraphH}.

\begin{figure}
    \labellist
    \small\hair 2pt
    \pinlabel $e_1$ at -10 270
    \pinlabel $e_2$ at 180 270
    \pinlabel $e_1$ at 235 270
    \pinlabel $e_2$ at 505 270
    \pinlabel $e_1$ at 115 70
    \pinlabel $e_2$ at 355 70
    \endlabellist
    \centering
    \includegraphics[scale = .8]{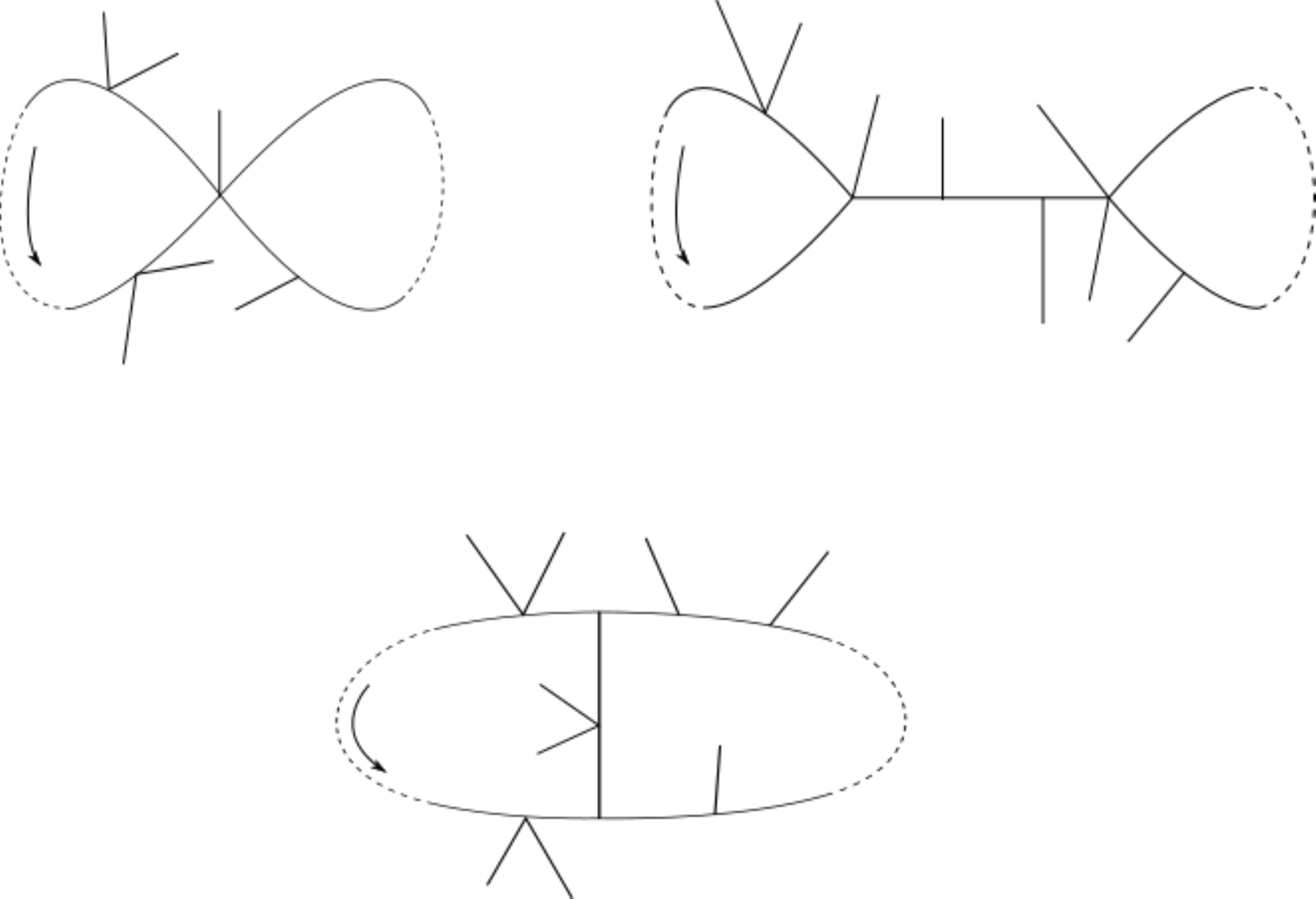}    \caption{Neighborhoods of the possible subgraphs $H$}
    \label{fig:subgraphH}
\end{figure}
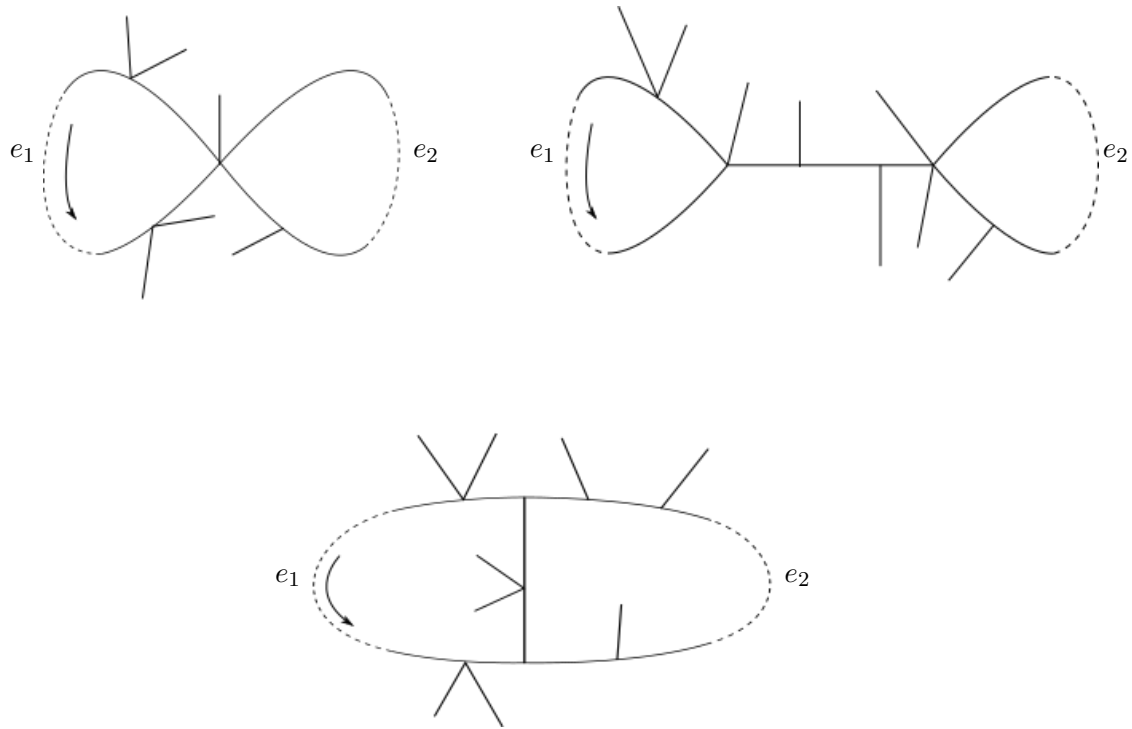

Define a configuration $x = (x_1, \ldots, x_n)$ as follows: Let $x_3, x_4, x_5, \ldots, x_{k+1}$ be points in the interior of $e_2, e_3, \ldots, e_k$, respectively. In the interior of $e_1$, place $x_2, x_{k+2}, \ldots, x_n$ equally spaced apart in $e_1$, such that $x_2$ is farthest back on $e_1$ with respect to the orientation on $e_1$, and let $A = [x_2, x_n]$. Finally, place $x_1$ on a branched vertex in $T \cap H$ as follows:
\begin{enumerate}[(i)]
    \item When $H = \infty$, then $x_1$ is the unique branched vertex of $H$.
    \item When $H = \Th$, then $x_1$ is the branched vertex of $H$ closer to the back end of $e_1$, such that when $x_1$ slides into $e_1$, we obtain a witness to $(1, 2)$.
    \item When $H = D$, then $x_1$ is the branched vertex of $H$ lying closer to $e_2$
\end{enumerate}

Further, let $B = G \sm (x \cup A)$, and label the components of $G \sm B$ by $X, Y, Z$ (and $X'$ when $H = \infty$) as indicated in Figure \ref{fig:labelledH}.

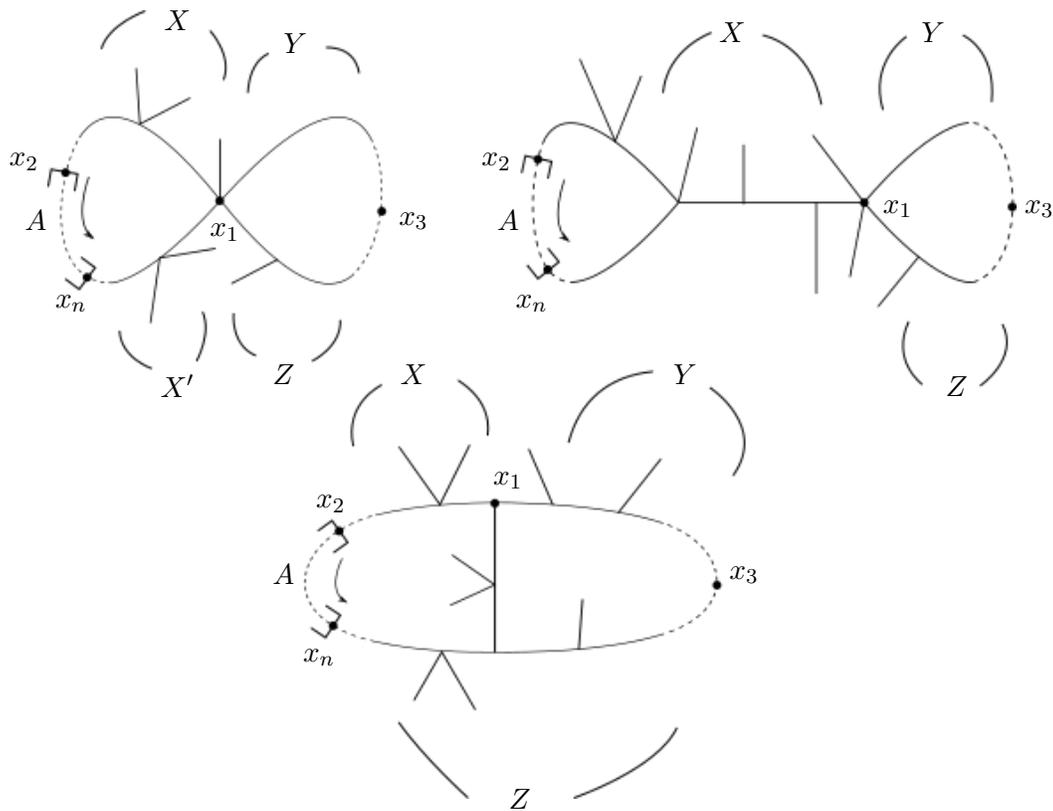
\begin{figure}
    \labellist
    \small \hair 2pt
    \pinlabel $A$ at -5 245
    \pinlabel $x_2$ at -10 270
    \pinlabel $x_{n}$ at 10 210
    \pinlabel $x_1$ at 75 240
    \pinlabel $X$ at 55 330
    \pinlabel $X'$ at 55 175
    \pinlabel $x_3$ at 155 245
    \pinlabel $Y$ at 105 320
    \pinlabel $Z$ at 100 180
    \pinlabel $A$ at 195 245
    \pinlabel $x_2$ at 190 270
    \pinlabel $x_{n}$ at 205 210
    \pinlabel $x_1$ at 360 250
    \pinlabel $x_3$ at 420 250
    \pinlabel $X$ at 290 325
    \pinlabel $Y$ at 375 325
    \pinlabel $Z$ at 385 175
    \pinlabel $A$ at 100 95
    \pinlabel $x_2$ at 120 125
    \pinlabel $x_{n}$ at 115 60
    \pinlabel $x_1$ at 195 135
    \pinlabel $x_3$ at 295 95
    \pinlabel $X$ at 155 180
    \pinlabel $Y$ at 270 180
    \pinlabel $Z$ at 200 0
    \endlabellist
    \centering
    \includegraphics[scale = .9]{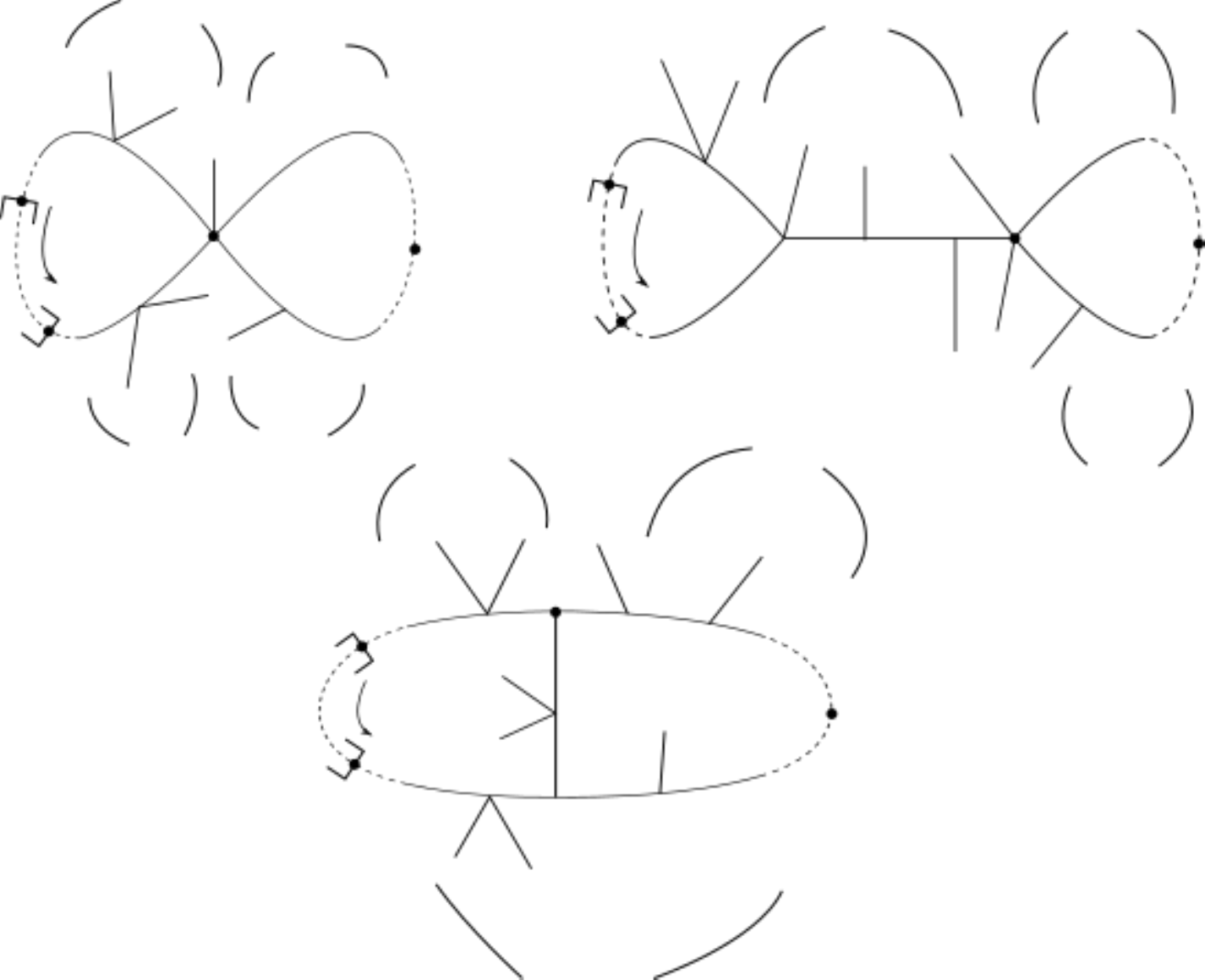}
    \caption{The components $X, Y, Z$, (and $X'$) of $G \sm B$.}
    \label{fig:labelledH}
\end{figure}

When we slide $x_1$ into the back of $e_1$, we obtain a witness to $(1, 2)$ by Proposition \ref{distpairprop}, and the reverse path gives a locally system of components by Lemma \ref{growingcomponents}, which implies exactly that $f(x) \in X$.

Let $x' = (x_3, x_2, x_1, x_4, \ldots, x_n)$ be the configuration which is $x$ except with $x_1$ and $x_3$ swapped. We will argue that $f(x')$ cannot lie in any component of $B$. If $f(x') \in X$ (or $X'$), then applying Lemma \ref{chasing} with $x_3$ produces a distinguished pair which does not contain $1$. Therefore we must have $f(x') \in Y$ or $f(x') \in Z$. 

We will define two paths $\be = (\be_1, \ldots, \be_n)$ and $\ga = (\ga_1, \ldots, \ga_n)$ in $\Conf_n(G)$ from $x$ to $x'$ which will allow us to see that neither case is possible. Take each of these paths to be concatenations of basic paths which fix each token except for the first and third, and such that $\be$ moves $x_1, x_3$ around the right loop in $H$ a half-turn \textit{clockwise}, and $\ga$ moves $x_1, x_3$ around the right loop in $H$ a half-turn \textit{counter-clockwise}\footnote{Here ``clockwise'' and ``counter-clockwise'' are with respect to the orientations implied by the figure}. 

Note that for each $t$, there are exactly two components of $G \sm \be(t)$ which border both $\be_1(t)$ and $\be_3(t)$. Since $\be_1, \be_3$ both move clockwise, $\be_1$, $\be_3$ are each moving ``towards'' a specific well-defined component at each $t$, and they likewise are each moving ``away'' from a specific well-defined component at each time. Further, for $t$ such that $\be_1(t) \neq x_1$, the path $\be_1$ is moving ``away'' from the component that $\be_3$ is moving ``towards.'' Our immediate goal is to show that $f(x') = f(\be(1))$ must lie in the component of $G \sm \be(1)$ which $\be_1$ is moving away from. We show this inductively on the basic pieces of $\be$.

Consider a Type I piece of $\be$. We can see that $f$ must remain in the component $\be_1$ is moving away from by Propositions \ref{typeiuniqueness} and \ref{principle1}.

Consider a Type II piece of $\be$ where only $\be_1$ moves to a vertex $v$. Since the component that $\be_1$ is moving away from is distinct from the components bordering $v$, we can deduce by Proposition \ref{principle2} that $f$ must remain in the component which $f$ is moving away from at the end of this basic piece.

Finally, consider a Type II piece of $\be$ where $\be_3$ moves, and suppose that $f$ takes values in a component of $G \sm \be$ which borders $\be_3$. Since we inductively assume that $f$ lies in the component of $G \sm \be$ which $\be_1$ is moving away from, $f$ must lie on the component of $G \sm \be$ that $\be_3$ is moving towards. Therefore, once $\be_3$ reaches the end of this Type II piece at some time $t$, $f$ must lie in some component $W$ of $G \sm \be(t)$ which borders $\be_3(t)$ but is not behind $\be_3$ at $t$ by Proposition \ref{principle2}. If $W$ does not border $\be_1(t)$, then by our construction of the configuration $x$, $W \cup \{\be_3(t)\}$ is simply connected, so we can apply Lemma \ref{chasing} to produce a distinguished pair which does not contain $1$. Therefore, $W$ must border $\be_1(t)$, so it must in fact be the component which $\be_3$ is moving towards, and therefore the component which $\be_1$ is moving away from.

We have shown that $f(x')$ is in the component which $\be_1$ is moving away from. Since $\be$ moves clockwise, we must have $f(x') \in Y$. However, if we apply the exact same argument with $\ga$, we get $f(x') \in Z$, which is a contradiction.
\end{proof}

\subsection{The Case $n = 3$}

When $n = 3$ and $\chi(G) = -1$, the statements (A) and (B) do not contradict Corollary \ref{distpaircor1}, so we must rely on a different technique. We first introduce a proposition which simplifies the proof.

\begin{prop}\label{defrettosbgrph}
Suppose $H \sbs G$ is a deformation retract of $G$, and suppose there exists an identifying function on $G$. Then there exists an identifying function on $H$.
\end{prop}

\begin{proof}
Since $H$ is a deformation retract of $G$, $G$ can be obtained from $H$ by attaching trees to $H$ at single points. Therefore, there exists a retract $r: G \ra H$ which collapses each of these trees to a single point. We define $g: \Conf_n(H) \ra H$ by $g(x) = r(f(x))$. We can see that $g$ is continuous, and we can show that it is also an identifying function. Suppose that $x = (x_1, \ldots, x_i)$ is a configuration, and suppose $x_i = r(f(x))$. Then $f(x) \in T$ for some tree $T$ attached to $G$. Corollary \ref{treechasing} produces a contradiction.
\end{proof}

Finally, we present another consequence of Proposition \ref{distpairprop} which will be useful to us.

\begin{prop}\label{distpaircor2}
If $(i, j)$ and $(k, i)$ are distinguished pairs on $e$, then $j = k$.
\end{prop}

\begin{proof}
Suppose otherwise, then we can construct a configuration $x = (x_1, \ldots, x_n)$ which is a simultaneous witness to $(i, j)$ and $(k, i)$, by picking $x_i, x_j, x_k$ all lying on $e$ with $x_k < x_i < x_j$ with respect to the orientation on $e$, and no other tokens contained in the interval $[x_k, x_j]$. By Proposition \ref{distpairprop}, such a configuration is a witness to both distinguished pairs, necessitating $f(x) \in (x_i, x_j)$ and $f(x) \in (x_k, x_i)$, which is a contradiction.
\end{proof}

We can now present the final portion of the proof of Theorem \ref{mainthm}.

\begin{proof}[Proof when $n = 3$ and $\chi(G) = -1$]
It suffices to consider this in the case where $G$ has no free vertices by Prop \ref{defrettosbgrph}. The only such graphs with $\chi(G) = -1$ are the graphs $\infty$, $\Th$, and $D$ as before. Suppose for sake of contradiction that $f$ is an identifying function on $G$.

We will define two functions 
\[\Hor: \Si_3 \ra \{\pm 1\},\] \[\Ver: \Si_3 \ra \{\pm 1\}.\] 
For a permutation $\si \in \Si_3$, consider the configuration $x_\si$ depicted in Figure \ref{fig:permconfs}.

\begin{figure}[h]
    \labellist
    \small\hair 2pt
    \pinlabel $x_{\si, 1}$ at 20 30
    \pinlabel $x_{\si, 2}$ at 75 25
    \pinlabel $x_{\si, 3}$ at 125 30
    \pinlabel $x_{\si, 1}$ at 190 30
    \pinlabel $x_{\si, 2}$ at 255 30
    \pinlabel $x_{\si, 3}$ at 295 30
    \pinlabel $x_{\si, 1}$ at 370 30
    \pinlabel $x_{\si, 2}$ at 430 40
    \pinlabel $x_{\si, 3}$ at 500 30
    \endlabellist
    \centering
    \includegraphics[scale = .8]{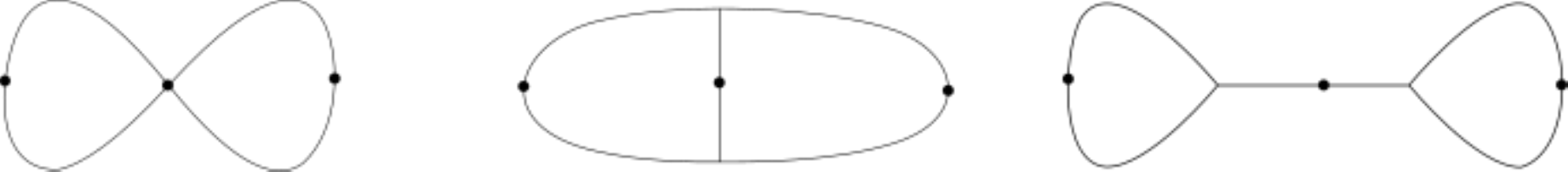}
    \caption{The configurations $x_{\si}$}
    \label{fig:permconfs}
\end{figure}

Denote by $x_{\si, i}$ the $i$'th token of $x_\si$. We consider the components of $G \sm x_\si$ to be either on the ``top'' or ``bottom,'' (resp. ``left,'' ``right'') corresponding to their depictions in the figure. When $G = \infty, \Th$, then we say $\Ver(\si) = +1$ if $f(x_\si)$ lies in one of the top components of $G \sm x_\si$, and $\Ver(\si) = -1$ otherwise. Similarly, when $G = \infty, D$ we say $\Hor(\si) = +1$ if $f(x_\si)$ lies in one of the right components of $G \sm x_\si$ and $\Hor(\si) = -1$ otherwise. To compute $\Ver$ for $G = D$, define $x'_\si$ to be the configuration obtained by moving along a Type II path which moves $x_{\si, 2}$ towards $f(x_\si)$. By Proposition \ref{principle2}, $f(x'_\si)$ lies in either the top or bottom component of $G \sm x'_\si$ contained in the loop containing $x_{\si, 2}$. We say that $\Ver(\si) = +1$ if $f(x'_\si)$ is in the top component of $G \sm x'_\si$ in this loop, and $\Ver(\si) = -1$ if $f(x'_\si)$ is in the bottom component of $G$. Similarly, for $G = \Th$, define $x'_\si$ to be the configuration obtained by moving along a Type II path which moves $x_{\si, 2}$ towards $f(x_\si)$. As before, we can conclude by Proposition \ref{principle2} that $f(x'_\si)$ lies on one of the components of $G \sm x'_\si$ in the direction that $x_{\si, 2}$ moved. We can therefore define $\Hor(\si) = +1$ if $f(x'_\si) \in (x'_{\si, 2}, x'_{\si, 3})$, and $\Hor(\si) = -1$ otherwise. Figure \ref{fig:verhor} shows a particular example of this construction.

$\,$

\begin{figure}[h]
    \labellist
    \pinlabel $x'_{\si, 1}$ at 15 25
    \pinlabel $x'_{\si, 2}$ at 80 40
    \pinlabel $f(x'_\si)$ at 110 35
    \pinlabel $x'_{\si, 3}$ at 155 25
    \pinlabel $x'_{\si, 1}$ at 205 25
    \pinlabel $x'_{\si, 2}$ at 320 35
    \pinlabel $f(x'_\si)$ at 365 65
    \pinlabel $x'_{\si, 3}$ at 395 25
    \endlabellist
    \centering
    \includegraphics{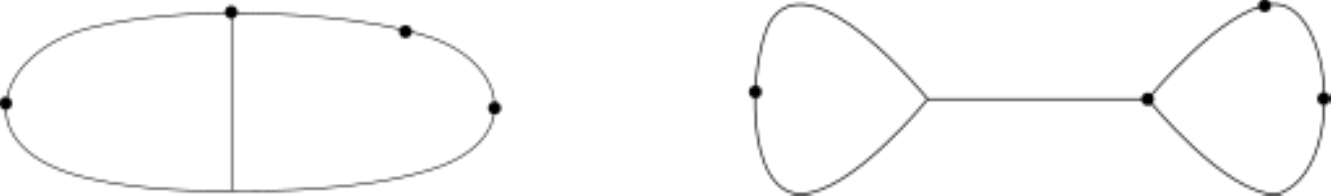}
    \caption{}
    \label{fig:verhor}
\end{figure}

Label the edges containing $x_{\si, 1}$ and $x_{\si, 3}$ by $e_1, e_3$ respectively. Note that $\Ver$ and $\Hor$ associate a distinguished pair to each $\si$ which contains $\si(2)$. Explicitly, this association is given by:

\[\si \mapsto (\si(2), \si(2 + \Hor(\si))) \text{ on } \Ver(\si) \cdot e_{2 + \Hor(\si)},\]

where $e_1, e_3$ are equipped with appropriate orientation, and where $\{\pm 1\}$ acts on the oriented edges of $G$ by changing orientation. Our next task is to determine how $\Ver$ and $\Hor$ interact with the multiplication in $\Si_3$. By possibly applying a symmetry to $G$, assume that $\Hor(\id) = \Ver(\id) = +1$. We make the following claims for all $\si \in \Si_3$:

\begin{align*}
    \Hor(\si) &= +1,\\
    \Ver((12)\si) &= \Ver(\si),\\
    \Ver((23)\si) &= -\Ver(\si).
\end{align*}

If we can prove these claims, then the remainder of the proof is simple, since we can write:
\[\Ver(\id) = \Ver((123)(123)(123)) = \Ver((12)(23)(12)(23)(12)(23)) = -\Ver(\id),\]

a contradiction.

All that remains is to prove the claim. We will prove this only in the case $G = \infty$, but the other cases are not so different.

We will show that $\Hor((12)) = \Hor((23)) = \Ver((12)) = +1$, and $\Ver((23)) = -1$. This will imply the rest of the claim by replacing $f$ with the identifying function $f_\si$, defined by
\[(x_1, x_2, x_3) \mapsto f(x_{\si(1)}, x_{\si(2)}, x_{\si(3)}).\]

Suppose the components of $G \sm x_{\id}$ are labelled as in Figure \ref{fig:figcpnts}

$\,$

\begin{figure}[h]
    \labellist
    \small\hair 2pt
    \pinlabel $X_1$ at 20 65
    \pinlabel $X_3$ at 15 10
    \pinlabel $X_2$ at 100 65
    \pinlabel $X_4$ at 95 10
    \endlabellist
    \centering
    \includegraphics{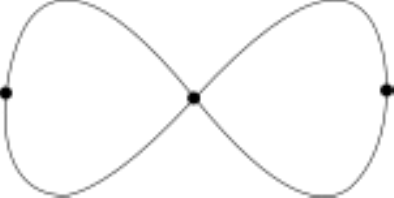}
    \caption{The components of $G \sm x_{\id}$}
    \label{fig:figcpnts}
\end{figure}

By assumption, we have $f(x_{\id}) \in X_2$. Consider the path $\al$ between $x_{\id}$ and $x_{(23)}$ which rotates $x_{\id, 2}$ and $x_{\id, 3}$ a half turn clockwise around the loop on the right. By representing $\al$ as a concatenation of a Type I and Type II path, we can see using Propositions \ref{principle1} and \ref{principle2} that $f(x_{(23)}) \in X_4$, so $\Ver((23)) = -1$ and $\Hor((23)) = +1$.

Now, consider the path $\be_1$ from $x_{\id}$ and $x_{(12)}$ which rotates $x_{\id, 1}$ and $x_{\id, 2}$ clockwise a half-turn around the right circle, and the path $\be_2$ which rotates them counter-clockwise a half-turn. By representing these as concatenations of a Type I and a Type II path, we can use Propositions \ref{principle1} and \ref{principle2} to show that $X_2 \not\lt_{\be_1} X_1$, and $X_2 \not\lt_{\be_2} X_3$. Finally, Proposition \ref{distpaircor2} shows that $f(x_{(12)}) \notin X_3$, otherwise $(2, 3), (3, 1)$ would both be distinguished pairs in $e_3$. Therefore, $f(x_{(12)}) \in X_2$, so $\Ver((12)) = \Hor((12)) = +1$, thereby establishing the claim, and concluding the proof.
\end{proof}

\section{A Combinatorial Model}\label{combmod}

In this section, we discuss $K_n(G)$, a combinatorial approximation to the spaces $\Conf_n(G)$ which was introduced in \cite{Luetgehetmann14}. These spaces are very useful for constructing sections in the cases when they exist, as we will show that partially defined identifying functions on $K_n(G)$ often extend to $\Conf_n(G)$. 

The outline of this section is as follows: we will first define a poset $P_n(G)$, which is defined slightly differently from the definition given in \cite{Luetgehetmann14}, but the two are easily seen to be equivalent. We will define $K_n(G)$ as the cube complex whose face poset is $P_n(G)$, and state without proof the fact that $K_n(G)$ embeds as a deformation retract of $\Conf_n(G)$. We will conclude by discussing which identifying functions on $K_n(G)$ extend to $\Conf_n(G)$.

Let $\E$ be the set of oriented edges of $G$, and let $\bar{\E} \sbs \E$ be the set of positively oriented edges, where we fix some positive orientation on each edge. Let $\B$ be the set of branched vertices of $G$. For each $e \in \E$, let $v_e$ be the endpoint of $e$ in the direction that $e$ is pointing, and let $\bar{e}$ be the element of $\bar{\E}$ corresponding to $e$. Recall that an \textit{index} denotes an integer $1, \ldots, n$. First we define the $k$-dimensional faces of $P_n(G)$.

\begin{defin}
A \textit{$k$-face} $F$ of $P_n(G)$ associates indices to the elements of $\bar{\E} \sqcup \B \sqcup \E$ as follows:
\begin{enumerate}[(i)]
    \item To each positively oriented edge $\bar{e} \in \bar{\E}$, there is an ordered tuple of associated indices:
    \[F(\bar{e}) = (i_{\bar{e},1}, \ldots, i_{\bar{e},\ell}).\]
    \item To each branched vertex $v \in \B$, there is \textit{at most} one associated index $F(v) = i_v$.
    \item To \textit{exactly} $k$ oriented edges $e \in \E$ such that $v_e \in \B$, there is an associated index $F(e) = i_e$.
    \item Each index occurs exactly once as an index associated to an element of $\bar{\E} \sqcup \B \sqcup \E$.
    \item For each vertex $v \in \B$, there is at most one index of the form $F(v)$ or $F(e)$, where $e \in \E$ is an oriented edge with $v_e = v$.
\end{enumerate}
\end{defin}

This gives a definition of $P_n(G)$ as a graded set. We now can define the partial ordering.

\begin{defin}
If $F$ is a $k$-face, and $e \in \E$ is an oriented edge such that $F(e)$ is defined, then we will define two $(k-1)$-faces, $F_e^+$ and $F_e^-$, which we will use to define the order on $P_n(G)$. First, we define $F_e^+$ as follows:
\begin{enumerate}[(i)]
    \item $F_e^+(\bar{e}') = F(\bar{e}')$ for all $\bar{e}' \in \bar{\E}$.
    \item $F_e^+(v_e) = F(e)$. For $v \in \B$ with $v \neq v_e$, $F_e^+(v)$ exists if and only if $F(v)$ does, and if they are defined, the two are equal.
    \item $F_e^+(e)$ is undefined, and for $e' \in \E$ with $e' \neq e$, $F_e^+(e')$ is defined if and only if $F(e)$ is, and if they are defined, the two are equal.
\end{enumerate}
Now we define $F_e^-$ as follows:
\begin{enumerate}[(i)]
    \item $F_e^-(\bar{e})$ is equal to $F(\bar{e})$ with $F(e)$ attached to the back when $e = \bar{e}$ (that is, when $e$ has positive orientation), and $F(\bar{e})$ with $F(e)$ attached to the front when $-e = \bar{e}$. For $\bar{e}' \in \bar{\E}$ with $\bar{e}' \neq \bar{e}$, define $F_e^-(\bar{e}') = F(\bar{e}')$.
    
    \item For $v \in \B$, $F_e^-(v)$ exists if and only if $F(v)$ does, and when they both exist the two are equal.
    
    \item $F_e^-(e)$ is undefined, and for $e' \in \E$ with $e' \neq e$, $F_e^-(e')$ exists if and only if $F(e')$ does, and when they both exist the two are equal.
\end{enumerate}
One can verify that $F_e^+$ and $F_e^-$ are indeed $(k-1)$-faces of $P_n(G)$. We define the ordering $\succ$ on $P_n(G)$ to be generated by the relations $F \succ F_e^+$ and $F \succ F_e^-$ for each face $F$ and each oriented edge $e$ such that $F(e)$ exists.
\end{defin}

The important facts about $P_n(G)$ are summarized in the following theorem, whose statements are proven in \cite{Luetgehetmann14}.

\begin{thm}\label{knembedding}
The graded poset $P_n(G)$ is the face poset of a unique (up to isomorphism) abstract cube complex $K_n(G)$. There is an embedding $i: K_n(G) \hra \Conf_n(G)$, and a deformation retract $h_t: \Conf_n(G) \ra \Conf_n(G)$ of $\Conf_n(G)$ onto the image of $K_n(G)$ in $\Conf_n(G)$. In addition, the following holds:
\begin{enumerate}[(a)]
    \item\label{evenlyspaced} The embedding $i$ takes vertices of $K_n(G)$ exactly to the configurations where the tokens are spread out evenly on each edge.
    
    \item\label{bitypei} For any configuration $x \in \Conf_n(G)$, no tokens of $x$ leave or enter any edges or vertices along the path $t \mapsto h_t(x)$. Alternatively, the paths $t \mapsto h_t(x)$ and $t \mapsto h_{1-t}(x)$ are both Type I for every $x$.
    
    \item\label{bdddist} Tokens of configurations in $K_n(G)$ are uniformly far from each other. That is, there exists some $\ep > 0$ such that for all $x = (x_1, \ldots, x_n) \in K_n(G)$, we have $d(x_i, x_j) \geq \ep$.
\end{enumerate}
\end{thm}

\begin{cor}\label{htydim}
If $G$ has $k$ branched vertices, then $\Conf_n(G)$ has the homotopy type of a $\min\{k, n\}$-dimensional cell complex.
\end{cor}

\begin{proof}
Indeed, one can check that $K_n(G)$ is $\min\{k, n\}$ dimensional. Details can be found in \cite{Luetgehetmann14}.
\end{proof}

% \begin{cor}\label{splittingcor}
% If $G$ has only one branched vertex, then the short exact sequence
% \begin{equation*}
%   1 \ra \ker(\phi_{n+1, *}) \ra PB_{n+1}(G) \xra{\phi_{n+1, *}} PB_n(G) \ra 1 
% \end{equation*}
% of groups, discussed in Section \ref{grouphom}, splits for every $n$.
% \end{cor}

% \begin{proof}
% By Corollary \ref{htydim}, $\Conf_{n}(G)$ has the homotopy type of a graph, so $PB_n(G)$ is a free group, and the sequence splits.
% \end{proof}

%Corollary \ref{splittingcor} provides, along with Proposition \ref{freevertsplit}, a class of examples of graphs such that the group homomorphism $\phi_{n+1, *}$ splits, although the map $\phi_{n+1}$ of topological spaces does not admit a section.

Intuitively, we can imagine the vertices of $K_n(G)$ to be the configurations of $\Conf_n(G)$ where the vertices in each edge are equally spaced, as suggested in Theorem \ref{knembedding}\ref{evenlyspaced}. Given such a configuration $x$, we discuss how to define the associated vertex $F$ of $P_n(G)$. For a positively oriented edge $\bar{e} \in \bar{\E}$, set $F(\bar{e})$ to be the tuple of indices whose corresponding tokens of $x$ lie in $\bar{e}$, ordered consistently with the orientation on $\bar{e}$. For a branched vertex $v \in \B$, set $F(v)$ to be the index of the token of $x$, if it exists, occupying $v$.

An edge in $K_n(G)$ corresponds to a continuous move between vertices of $K_n(G)$ which moves a single token from a vertex into the interior of an edge. Similarly, a $k$-dimensional face corresponds to a set of $k$ moves which can be performed consistently, that is, they involve distinct vertices and tokens.

In our construction of the poset $P_n(G)$, for a fixed face $F$, the indices associated with positively oriented edges $\bar{e}$ and vertices $v$ correspond to tokens which are essentially fixed on the image of $F$, and the location of the token corresponding to such an index is determined by which positively oriented edge or vertex it is associated to. Conversely, the indices associated to oriented edges $e$ correspond to tokens which can move from the interior of $e$ to the vertex $v_e$ throughout $F$. For such an $e$, the associated face $F_e^+$ corresponds to the boundary face of $F$ where the moving token on $e$ is instead taken to be stationary at $v_e$, and $F_e^-$ corresponds to the boundary face where the moving token on $e$ is taken to be stationary on the interior of $e$.

\begin{example}
Let $G$ be the dumbbell graph pictured in Figure \ref{fig:dumbbell}, with the given labels and choices of positive orientations, and let $n = 5$.

\begin{figure}[h]
    \labellist
    \small\hair 2pt
    \pinlabel $e_1$ at -10 30
    \pinlabel $b_1$ at 55 40
    \pinlabel $e_3$ at 80 35
    \pinlabel $b_2$ at 110 40
    \pinlabel $e_2$ at 175 30
    \endlabellist
    \centering
    \includegraphics[scale = .8]{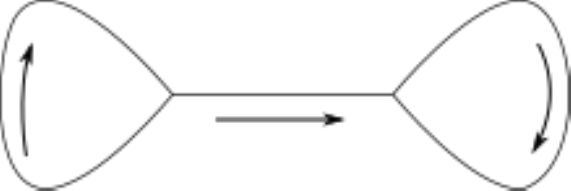}
    \caption{}
    \label{fig:dumbbell}
\end{figure}

A vertex of $K_n(G)$ maps to a configuration like $x$, pictured in Figure \ref{fig:dumbvert}, such that the vertices on each edge are evenly spaced apart.

\begin{figure}[h]
    \labellist
    \small\hair 2pt
    \pinlabel $x_1$ at -10 30
    \pinlabel $x_2$ at 110 40
    \pinlabel $x_5$ at 45 55
    \pinlabel $x_4$ at 180 30
    \pinlabel $x_3$ at 45 5
    \endlabellist
    \centering
    \includegraphics[scale = .7]{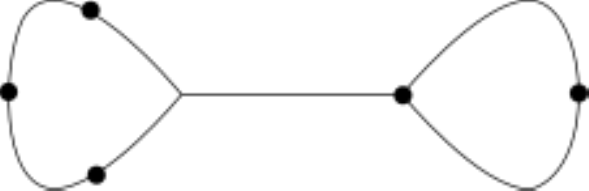}
    \caption{}
    \label{fig:dumbvert}
\end{figure}

Let $F_x$ be the $0$-face which maps to $x$. Explicitly, we can compute:
\begin{align*}
    F(\bar{e_1}) &= (3, 1, 5), \\
    F(\bar{e_2}) &= (4), \\
    F(b_2) &= 2,
\end{align*}
and every other value of $F$ is trivial.
\end{example}

We now turn to the question of extending identifying functions defined on $K_n(G)$ to $\Conf_n(G)$.

The important fact is that under certain mild conditions, we can extend any identifying function defined on $K_n(G)$ to all of $\Conf_n(G)$.

\begin{defin}\label{extendable}
Suppose $A \sbs K_n(G)$ is a subset containing the $0$-skeleton of $K_n(G)$, and $f: A \ra G$ is a partially defined identifying function. We say that $f$ is \textit{extendable} if the following situation never occurs: Select a free edge $e$ (edge connected to a free vertex) which is connected to a branched vertex $v$. Then there exists a vertex $F$ of $K_n(G)$ where $F(v)$ exists, $F(\bar{e})$ is empty, and $f(F) \in e$.
\end{defin}

\begin{prop}\label{extextends}
A partially defined identifying function $f: K_n(G) \ra G$ extends to an identifying function on $\Conf_n(G)$ if and only if $f$ is extendable.
\end{prop}

\begin{proof}
Suppose $f$ is extendable, and let $x = (x_1, \ldots, x_n)$ be a configuration, let $y = (y_1, \ldots, y_n) = h_1(x)$, and let $p = f(y)$. If $p$ is a branched vertex of $G$, then define $f(x) = p$. Note that $p \neq x_1, \ldots, x_n$, since the path $t \mapsto h_{1-t}$ is Type I by Theorem \ref{knembedding}\ref{bitypei}. If $p$ is not a branched vertex of $G$, then $p$ lies in some component $A$ of $G \sm (h_1(x) \cup \B)$. Note that $A$ must be either: an open interval $(a_1, a_2)$ or a half-open interval $(a_1, a_2]$, where $a_2$ is a free vertex of $G$. Let $V$ be the set of (branched or free) vertices of $G$. Define a function 
\[\psi: \{y_1, \ldots, y_n\} \sqcup V \ra G\]
by:
\[\psi(q) = 
\begin{cases} 
x_j & q = y_j \\
q & q \in V
\end{cases}.\]

We would like to define $f(x)$ to lie in the open interval $(\psi(a_1), \psi(a_2))$, where we take $a_1, a_2 \in \{y_1, \ldots, y_n\}$ if possible. 

First we must check that $(\psi(a_1), \psi(a_2))$ is nonempty. If $\psi(a_1) = \psi(a_2)$, then we must have that $a_2$ is a free vertex of $G$ and $a_1 = y_j$ for some $j$. Then, there is a path in $\Conf_n(G)$ from $y$ to a vertex $F$ of $K_n(G)$ such that $y_j$ moves directly onto the nearest branched vertex $v$. By Lemma \ref{growingcomponents} and Proposition \ref{principle2}, if $f$ extends, we must have $f(F) \in [a_2, v)$, which contradicts the fact that $f$ is extendable. Therefore $(\psi(a_1), \psi(a_2))$ is a nonempty interval. 

We define $f(x)$ to lie between $\psi(a_1), \psi(a_2)$ in the same proportion that $f(y)$ lies between $a_1, a_2$. This defines $f$ on $\Conf_n(G)$, and it is easily verified to be a continuous identifying function.

Now suppose $f$ is not extendable. Note that the situation in Definition \ref{extendable} contradicts Corollary \ref{treechasing}, so $f$ does not extend to an identifying function on $\Conf_n(G)$.
\end{proof}

\begin{cor}
If $G$ has no free vertices, then every identifying function on $K_n(G)$ extends to $\Conf_n(G)$.
\end{cor}

\section{Sharpness of $n \geq 2 - \chi(G)$}

In this section we use the methods of Section \ref{combmod} to show that we cannot make any stronger claims about the existence or nonexistence of sections based only on the homotopy type of a graph, that is, whenever $n < 2 - \chi(G)$ and $\chi(G) < 0$, it is impossible to determine whether $G$ admits an identifying function. First, we prove an existence result.

\begin{prop}
Suppose $G$ contains a vertex $w$, such that $G \sm w$ contains $k$ components which are connected to $w$ at least twice. Then, if $n \leq k$, $G$ admits an identifying function.
\end{prop}

\begin{proof}

First, we consider the graph $H = \bigvee_{1}^k S^1_k$, a wedge of $k$ circles, where $k \geq n$. Let $v$ be the unique branched vertex, and let $p_1, \ldots, p_k$ be the points such that $p_i$ is the point of $S_i^1$ antipodal to $v$. Let $S$ be the set $\{1, 2, \ldots, k\}$, and let $\F$ be the class of functions $\psi: S \ra \NN$ such that
\[\sum_{s \in S} \psi(s) = n - 1.\]
Let $h: \F \ra S$ be a function such that $\psi((h(\psi))) = 0$ for all $\psi \in \F$. Since $k \geq n$, such a function always exists.

Select some small $\ep > 0$, and let $M \sbs H^n$ be the subset:
\[M = \{(x_1, \ldots, x_n): x_i \in H, \text{ there is at most one $i$ such that } d(x_i, v) < \ep \}.\]
(note that we do not require the $x_i$ to be distinct). We will define a $\Si_n$-invariant identifying function on $M$, that is, a function $f: M \ra H$ such that $f(x_1, \ldots, x_n) \neq x_i$ for all $x = (x_1, \ldots, x_n) \in M$, and $i = 1, \ldots, n$ such that $f(\si x) = f(x)$, where $\Si_n$ acts on $M$ by permuting indices. Let $N \sbs H$ be the $\ep$-neighborhood of $v$, and let $A = (G \sm N)^n$ and $B = N \ti (G \sm N)^{n-1}$. Note that $A$ and $B$ are disjoint, and every element of $M$ is in the $\Si_n$-orbit of an element in $A$ or $B$ (but not both), so it suffices to define $f$ on $A \cup B$. On $A$, we simply define $f$ to be constantly equal to $v$. On $B$, we will define $f$ separately on each component $X$ of $B$. For each $X$, we can define a function $\psi_X: S \ra \NN$ by setting $\psi_X(i)$ to be the number of tokens of an element of $X$ contained in $S_i^1 \sm N$. Clearly, $\psi_X \in \F$. Let $\pi: B \ra N$ be the projection of $B$ onto $N$. We can define a map $\lambda: \bar{N} \ra S^1_{h(\psi_X)}$ such that:
\begin{enumerate}[(a)]
    \item $\lambda(v) = p_{h(\psi_X)}$
    \item For each free vertex $u$ of $\bar{N}$, $\lambda(u) = v$
    \item $\lambda$ has no fixed points, that is, $\lambda(q) \neq q$ for all $q \in N \cap S^1_{h(\psi_X)}$.
\end{enumerate}
By the definition of $h$, $\psi_X(h(\psi_X)) = 0$, so $S_{h(\psi_X)}$ contains no tokens of any element of $X$ by the definition of $\psi_X$. Therefore, such a map exists. We define $f = \lambda \circ \pi$ on $X$. This is an identifying function on $X$, and by applying this process to all components of $B$, we get an identifying function on $B$. Further, by taking $f$ to be $\Si_n$ invariant, we get an identifying function on $M$, as desired, by noting that the local definitions of $f$ glue continuously.

Now, we turn to the task of defining an identifying function on $G$. By our assumptions on $w$, there exists an embedding $i: H \hra G$ such that $i(v) = w$, and a retract $r: G \ra H$ such that $r^{-1}(w) = \{v\}$. Note that $r$ induces a map $R: K_n(G) \ra H^n$, by applying $r$ to each token of a configuration. By Theorem \ref{knembedding}\ref{bdddist}, there is some $\ep$ such that at most one token of any configuration of $K_n(G)$ is within $\ep$ of $w$. Taking sufficiently small $\ep$, we can see that this implies that $R(K_n(G)) \sbs M$. Then, let $g = i \circ f \circ R$, where $f$ is as constructed above. For any $x \in K_n(G)$, note that none of the tokens of $x$ lie in the image of the circle which contains $g(x)$, so $g$ is a partially defined identifying function. Further, the image of $H$ in $G$ intersects no free edges, so $g$ is extendable, and by Proposition \ref{extextends}, we conclude.

\end{proof}

Now we present a class of examples with $n < 2 - \chi(G)$ where sections do not exist.

\begin{example}[Wedge of balloons]
Let $B$ be the balloon graph, as pictured in Figure \ref{fig:balloon}

\begin{figure}[h]
    \centering
    \includegraphics[scale = .7]{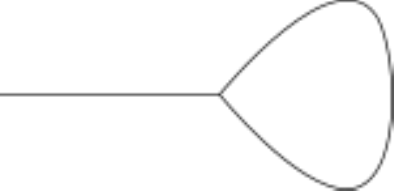}
    \caption{}
    \label{fig:balloon}
\end{figure}

Suppose $G = \bigvee_1^k B$ is a wedge of balloons, glued at the free vertices of each balloon, and suppose that $n \leq k$. If $n \geq 2$, then statement (A) holds, and if $n \geq 3$, statement (B) holds, and the proofs of these are quite similar to the proof of Theorem \ref{mainthm}. The key observation is that whenever we have a configuration $x$ with $f(x)$ in a copy of $B$ and a token (say $x_1$) as depicted in Figure \ref{fig:initballoon}, there always exists a distinguished pair $(1, j)$ in the circle in that copy of $B$ (in fact for any $j \neq 1$, possibly after changing orientation).

$\,$

\begin{figure}[h]
    \labellist
    \small\hair 2pt
    \pinlabel $x_1$ at 50 40
    \pinlabel $f(x)$ at 95 70
    \endlabellist
    \centering
    \includegraphics[scale = .7]{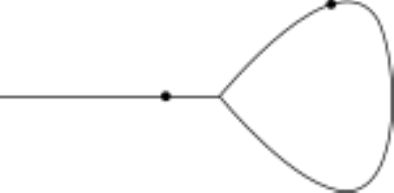}
    \caption{}
    \label{fig:initballoon}
\end{figure}

To see why, note that we can slide any token $x_j$ into $B$, and continue through the path depicted by Figure \ref{fig:balloonmove}.

\begin{figure}[h]
    \labellist
    \small\hair 2pt
    \pinlabel $x_1$ at 50 185
    \pinlabel $f(x)$ at 95 210
    \pinlabel $x_j$ at 20 110
    \pinlabel $x_1$ at 60 110
    \pinlabel $f(x)$ at 120 135
    \pinlabel $x_j$ at 60 40
    \pinlabel $f(x)$ at 125 65
    \pinlabel $x_1$ at 130 30
    \endlabellist
    \centering
    \includegraphics[scale = .7]{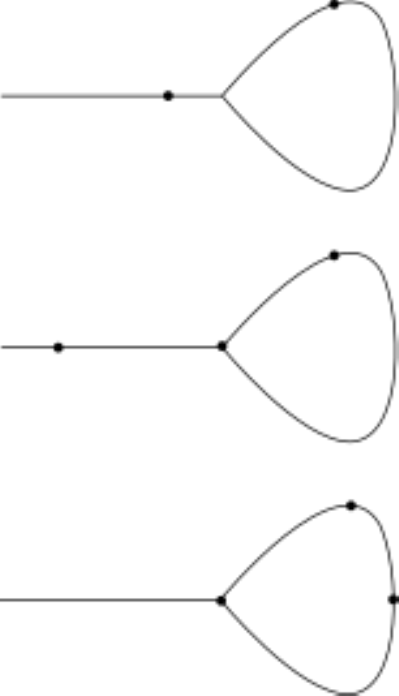}
    \caption{}
    \label{fig:balloonmove}
\end{figure}

To prove (A) and (B), note that we can avoid placing tokens in any balloons we would like, and copy the proof of Theorem \ref{mainthm}, and the fact above when necessary.
\end{example}

Therefore we have shown that when $n < 2 - \chi(G)$ (and $n \geq 4$), we cannot say for sure whether or not $G$ admits a section.

\bibliographystyle{alpha}
\bibliography{sectprobs}
\end{document}